\documentclass[final]{siamart190516}
\usepackage[english]{babel}
\usepackage[utf8]{inputenc} 
\usepackage[T1]{fontenc}	
\usepackage{amsfonts}		
\usepackage{mathrsfs}		
\usepackage{dsfont}			
\usepackage{yfonts}
\usepackage{pdfsync} 
\usepackage{amssymb, amsmath}

\usepackage{verbatim} 
\usepackage{todonotes} 
\usepackage{color}
\usepackage{url}
\usepackage{thmtools}
\usepackage[numbers]{natbib} 
\usepackage{enumitem} 
\usepackage{algorithmic}
\usepackage{graphicx} 
\usepackage[percent]{overpic}
\usepackage{subcaption}

\definecolor{darkblue}{rgb}{0.0,0.0,0.6}

\newsiamremark{remark}{Remark}

\newsiamremark{assumption}{Assumption}

\usepackage{my_latex_commands}

\newcommand{\overbar}[1]{\mkern 1.5mu\overline{\mkern-1.5mu#1\mkern-1.5mu}\mkern 1.5mu}
\newcommand{\ind}{{\tau}}
\newcommand{\zth}[1]{z_{#1}}
\newcommand{\tth}[1]{t_{#1}}
\newcommand{\zcu}[1]{\overline{z}_\tau^{#1}}
\newcommand{\tcl}[1]{\underline{t}_\tau^{#1}}

\newcommand{\V}{\mathbb{V}}
\newcommand{\ovbar}[1]{\overline{#1}}
\newcommand{\ubar}[1]{\underline{#1}}

\DeclareMathAlphabet{\mathpgoth}{OT1}{pgoth}{m}{n}

\DeclareFontFamily{U}{jkpmia}{}
\DeclareFontShape{U}{jkpmia}{m}{it}{<->s*jkpmia}{}
\DeclareFontShape{U}{jkpmia}{bx}{it}{<->s*jkpbmia}{}
\DeclareMathAlphabet{\mathfrak}{U}{jkpmia}{m}{it}
\SetMathAlphabet{\mathfrak}{bold}{U}{jkpmia}{bx}{it}
\newcommand{\kindex}{\mathfrak{k}}

\begin{document}


\title{A-priori error analysis of local incremental minimization schemes for rate-independent
evolutions
\thanks{
\funding{This research was supported by the German Research Foundation (DFG) under grant 
number~HE~6077/8-1 within the priority program Non-smooth and Complementarity-based
Distributed Parameter Systems: Simulation and Hierarchical Optimization (SPP~1962).}}}

\author{Christian Meyer\thanks{Technische Universit\"at Dortmund, Fakult\"at f\"ur Mathematik, 
 \email{cmeyer@math.tu-dortmund.de}  }
\and Michael Sievers\thanks{Technische Universit\"at Dortmund, Fakult\"at f\"ur Mathematik, 
 \email{michael.sievers@math.tu-dortmund.de} } }
 
\headers{Error analysis of local incremental minimization schemes}{Christian Meyer and Michael Sievers}

%
%
%


\maketitle

\begin{abstract}
This paper is concerned with a priori error estimates for the local incremental minimization scheme, which is 
an implicit time discretization method for the approximation of rate-independent systems with non-convex energies.
We first show by means of a counterexample that one cannot expect global convergence of the scheme without any 
further assumptions on the energy. For the class of uniformly convex energies, we derive error estimates of optimal order, 
provided that the Lipschitz constant of the load is sufficiently small. Afterwards, we extend this result to the case of an energy,
which is only locally uniformly convex in a neighborhood of a given solution trajectory.
For the latter case, the local incremental minimization scheme turns out to be superior compared to its global counterpart, 
as a numerical example demonstrates.
\end{abstract}

\begin{keywords}
rate independent evolutions, incremental minimization schemes, a priori error analysis, implicit time discretization, 
parameterized solutions, differential solutions
\end{keywords}
\begin{AMS}
65J08, 65K15, 65M15, 74C05, 74H15
\end{AMS}

\section{Introduction}\label{sec:intro}

This paper is concerned with a-priori error estimates for the numerical approximation of rate-independent processes. 
The system under investigation is of the form 
\begin{equation} \label{eq:subDiffInc}
0 \in \partial \RR(z^\prime(t)) + D_z\II(t,z(t)) \quad \ae \text{ in } [0,T], \tag{RIS}\end{equation}
where $\II$ denotes the energy functional and $\RR$ is a positive 1-homogeneous dissipation. 
The precise assumptions on the data are given in Section~\ref{sec:data}  below. 
The rate-independence manifests itself through the 1-homogeneity of the dissipation, 
which in fact induces that the system is invariant under time-rescaling. 
This simply means that rescaling the time in \eqref{eq:subDiffInc} results in a likewise rescaled solution.

By now, there exists a variety of different solution concepts for \eqref{eq:subDiffInc} 
being capable of handling time-discontinuities, which may occur due to non-convexity of the energy functional. 
We refer to \cite{mielkeroubi} for an overview. In this paper, we focus on the notion of \emph{parameterized solutions}. 
Loosely speaking, the main idea behind this solution concept is to parameterize the graph of an evolution satisfying
\eqref{eq:subDiffInc} by arc-length. The process is thus described in an artificial time $s$ by the following system 
\begin{equation} \left\{\qquad
\begin{gathered}
t(0) = 0, \quad z(0) = z_0, \quad t^\prime(s) + \norm{z^\prime(s)} = 1, \\
0 \in \partial \RR(z^\prime(s)) + \lambda(s)z^\prime(s) + D_z\II(t(s),z(s)) \\
\lambda(s) \geq 0, \quad \lambda(s)(1-\norm{z^\prime(s)}) = 0, 
\end{gathered} \right. 
\label{param-sol-lam} 
\end{equation} 
see \cite{efenmielke06,mielkeroubi} for details. 
Existence of solutions in the sense of \eqref{param-sol-lam} can be established in multiple ways, for instance 
by means of a vanishing viscosity analysis, see e.g.~\cite{mrs12, mrs16}. 

Another approach to show existence 
is to apply particularly chosen time discretization schemes and pass to the limit with the time step size.
A prominent example for this procedure is the so-called \emph{local incremental minimization scheme} of the form
\begin{subequations}\label{eq:locminscheme1}
\begin{align} 
 z_{k} &\in \argmin\{\II(t_{k-1},z) + \RR(z-z_{k-1}) \, : \, z \in \ZZ, \, \norm{z-z_{k-1}}_\V \leq \tau\}  \label{eq:locmin1}\\
 t_{k} &= \min\{t_{k-1}+\tau-\norm{z_{k}-z_{k-1}}_\V,T\} . \label{eq:tupdate1}
\end{align}
\end{subequations}
This approach is for instance pursued in \cite{efenmielke06} for the finite dimensional 
and in \cite{Neg14, knees17} for the infinite dimensional case. 
The authors show (weak) convergence of subsequences to solutions of \eqref{param-sol-lam} as $\tau \searrow 0$.
In \cite{fem_paramsol}, a finite element discretization is incorporated into the convergence analysis.
Moreover, as also demonstrated in \cite{fem_paramsol}, the scheme in \eqref{eq:locminscheme1} is not only 
interesting from a theoretical point of view, but can also be efficiently realized in practice 
for instance by means of a semi-smooth Newton method.
Let us mention that there exist other discretization methods to approximate parameterized solutions, 
such as relaxed local minimization schemes as proposed in \cite{acfs17} or alternating minimization schemes, 
if a second variable enters the energy functional. Moreover, time discretization and viscous regularization can be coupled 
to approximate a parameterized solution, see \cite{ks13, mrs16}. For a detailed overview, we refer to \cite{knees17}.

However, when it comes to rates of convergence for discretizations using \eqref{eq:locminscheme1}, 
the literature becomes rather scarce. Since, in case of non-convex energies, the (parameterized) solution of \eqref{eq:subDiffInc} 
is in general not unique, not even locally, as there might be a whole continuum of solutions, one can in general hardly expect any 
a priori estimates. The situation changes, if one turns to \emph{uniformly convex energies}. 
In this case, however, there is no need for a localized scheme as in \eqref{eq:locminscheme1} so that one can drop the additional constraint 
in \eqref{eq:locmin1} and simply use the a time-update of the form  $t_{k} = t_{k-1}+\tau$. 
The method arising in this way is called \emph{global incremental minimization scheme} and can be shown to converge to the 
\emph{global energetic solution}, which is unique in case of a uniformly convex energy.
Even more, in \cite{mielketheil,mpps:errorRIS}, the authors show that the error between the discrete solution of this scheme
and the global energetic solution is of order $\OO(\sqrt{\tau})$.
This result has been improved in \cite{mielke:ERIS} and, more generally, in \cite{bartels:errorEst} to rates of order $\Landau{\tau}$
for the case of a quadratic and coercive energy. 
An energy functional with these properties arises for instance in case of quasi-static elastoplasticty with linear kinematic hardening, 
where several convergence results have been obtained by various authors, see e.g.~\cite{hanreddy, AC:NumAnElastoPlast} and the references therein. 
Recently, in \cite{RinSchwarzSueli17}, the authors provide an a priori error estimate for the global minimization scheme 
in case of a semilinear and uniformly convex energy including a spatial discretization.

By contrast, to the best of our knowledge, there exists no such convergence results 
for the local incremental minimization scheme in \eqref{eq:locminscheme1}, even not in the case of a uniformly convex energy. 
With the present paper, we aim to fill this gap. Moreover, 
we provide an a priori estimate, if the energy functional is only \emph{locally uniformly convex} along a given solution trajectory. 
At this point, the local incremental minimization scheme turns out to be superior to the global one, since the latter does in general not satisfy
such an a priori estimate as we will demonstrate by means of a counterexample. 
In summary, the overall picture concerning the local incremental minimization scheme now looks as follows:
\begin{itemize}
    \item For an arbitrary non-convex energy, there exists a subsequence of discrete solutions that converges (weakly) to a parameterized solution
    as $\tau \searrow 0$.
    \item If the energy is locally uniformly convex along a solution trajectory, then the discrete solution converges with optimal rate 
    to this solution, provided that the time step size is sufficiently small.
    \item If the energy is uniformly convex, one obtains the same convergence rates as for the global incremental minimization scheme.
\end{itemize}


The paper is organized as follows. In Section~\ref{sec:assumptions}, we lay the foundations for our 
a priori error analysis. We present our standing assumptions, the solution concepts for \eqref{eq:subDiffInc} underlying 
our analysis, and the local incremental minimization scheme in a rigorous manner. 
The section ends with a simple one-dimensional example which shows that one can indeed not expect 
any convergence result for the whole sequence of discrete solutions without any further assumption on the energy 
such as (local) uniform convexity.
The third section is then devoted to the derivation of our a priori estimates. 
In the first subsection, we provide some basic estimates that are frequently used throughout the convergence analysis. 
In Sections~\ref{sec:GlobConvLip} and \ref{sec:GlobConvGen}, it is assumed that the energy is (globally) uniformly convex.
We start our a priori analysis with an additional assumption saying that the driving force is Lipschitz continuous with a sufficiently small Lipschitz constant. 
In Section~\ref{sec:GlobConvGen}, we then drop the smallness assumption on the Lipschitz constant. 
It is to be noted that, in this case, we do not obtain the optimal order of convergence, see Remark~\ref{rem:notoptimal} below.
Finally, Section~\ref{sec:LocConv} is concerned with the a priori analysis in case of locally uniformly convex energies. 
The numerical experiments in Section~\ref{sec:tests} illustrate our theoretical findings.

\section{Notation and standing assumptions}\label{sec:assumptions}

Let us start with some basic notation used throughout the paper. 
Unless indicated, $C>0$ always is a generic constant. 
Moreover, given two normed linear spaces $X, Y$, we denote by $\dual{\cdot}{\cdot}_{X^*, X}$ 
the dual pairing and suppress the subscript, if there is no risk for ambiguity. 
By $\|\cdot\|_X$, we denote the norm in $X$ and $\LL(X,Y)$ is the space of linear and bounded 
operators from $X$ to $Y$. Furthermore, $B_X(x,r)$ is the open ball in $X$ around $x\in X$ 
with radius $r>0$.

\subsection{Assumptions on the data}\label{sec:data}

Let us now introduce the assumptions on the quantities in \eqref{eq:subDiffInc}. 

\paragraph{Spaces}
Throughout the paper, $\XX $ is a Banach space and $\ZZ,\VV$ are Hilbert spaces such that
$\ZZ \overset{c,d}{\embeds} \VV \embeds \XX$,
where $\overset{d}{\embeds}$ and $\overset{c}{\embeds}$ refer to dense and compact embedding, respectively. For convenience, we will assume w.l.o.g. that the embedding constant $c_\ZZ$ of $\ZZ \to \VV$ fulfills $c_\ZZ =1$. Otherwise only the constants in the corresponding estimates will change. 
For the same reason, we will use the natural norm in $\VV$ rather than an equivalent one as carried out in \citep{knees17}. 
The Riesz isomorphism associated with $\VV$ is denoted by $J_\VV: \VV \to \VV^*$. 

\paragraph{Energy}
For the energy functional we require that $\II$ has the following semilinear form:
\begin{equation*} 
\II : [0,T] \times \ZZ \to \R, \quad \II(t,z)=\frac{1}{2}\dual{A z}{z}_{\ZZ^\ast,\ZZ} + \FF(z) - \dual{\ell(t)}{z}_{\VV^*, \VV} \, .
\end{equation*}
wherein $A \in \LL(\ZZ, \ZZ^*)$ is a self-adjoint and coercive operator, i.e., there is a constant $\alpha > 0$ 
such that $\dual{Az}{z}_{\ZZ^*, \ZZ} \geq \alpha \|z\|_{\ZZ}^2$. In addition, we assume that 
$\ell \in C^{0,1}([0,T];\VV^*) $ and $\FF \in C^2(\ZZ;\R)$ with $\FF \geq 0$ 
and write $| \ell |_{Lip}$ for the Lipschitz constant. The restriction of $\ell(\cdot)$ to a functional on $\ZZ$ is, for convenience, denoted by the same symbol. 

For the non-quadratic part, we assume that $\FF$ is of lower order compared to $A$ which means that
\begin{equation}\label{eq:assuF}
 D_z \FF \in C^1(\ZZ,\VV^\ast) , 
 \quad \norm{D^2_z \FF(z)v}_{\VV^\ast} \leq C_\FF (1+\norm{z}_\ZZ^q)\norm{v}_\ZZ 
\end{equation}
for some $q \geq 1$ so that, for every $z\in \ZZ$, $D_z\FF(z)$ can uniquely be extended 
to a bounded and linear functional on $\VV$, which we again denote by the same symbol for convenience. 

Moreover, we additionally assume that $\II(t,\cdot) \in C^{2,1}_{loc}(\ZZ;\R)$, 
that is to say, for all $r>0$ there exists $C(r) \geq 0$ such that for all $z_1,z_2 \in B_\ZZ(0,r)$ it holds
\begin{equation}\label{ass:regularity_I}  
\dual{\big[D_z^2\II(t,z_1)-D_z^2\II(t,z_2)\big]v}{v}_{\ZZ^\ast,\ZZ} \leq C(r) \norm{z_1-z_2}_\ZZ \norm{v}_\ZZ^2 . 
\end{equation}
Note that, due to the structure of the energy functional $\II$, the constant $C(r)$ does not depent on the time $t$ and, moreover, this assumptions holds iff $\FF \in C^{2,1}_{loc}(\ZZ;\R)$.
Lastly, we require $\II$ to be (at least locally) uniformly convex, see \cref{ass:globconv} and \cref{ass:locconv} below, which we will indicate at the appropriate places.




\paragraph{Dissipation} 
In the following, we denote by $\RR$ the dissipation potential and assume $\RR : \VV \to [0,\infty)$ 
to be lower semicontinuous, convex, and positively homogeneous of degree one. 
Moreover, we require the dissipation to be bounded, i.e., there exist constants $\underline{\rho},\overline{\rho}>0$ such that, 
for all $v\in \VV$ there holds $\underline{\rho} \norm{v}_\XX \leq \RR(v) \leq \overline{\rho} \norm{v}_\VV$.  
Since $\RR$ is convex and l.s.c., it is locally Lipschitz continuous so that its subdifferential is bounded for every point of the domain. 

\paragraph{Initial data} Finally we assume that the initial state $z_0$ satisfies $z_0 \in \ZZ$ and 
$0 \in \partial \RR(0) + D_z \II(0,z_0)$, i.e., $z_0$ is locally stable.


\subsection{Solution Concepts}

We now turn to our notion of solutions and give a rigorous definition thereof. 
For a broad overview over the various solution concepts for rate independent systems, we refer to \cite{mielke11differential,mielkeroubi} and the references therein.

\begin{definition}\label{def:diffsol}
We call $z : [0,T] \to \ZZ$ a \emph{differential solution} of the rate-independent system \eqref{eq:subDiffInc}, 
if $z \in W^{1,1}(0,T;\ZZ)$ with $z(0) = z_0$ and $0 \in \partial\RR(z^\prime(t)) + D_z\II(t,z(t))$ f.a.a.\ $t \in [0,T]$. 
\end{definition}

Due to the $1$-homhogeneity of $\RR$, it holds $\partial\RR(v) \subset \partial\RR(0)$ for all $v \in \VV$. 
Thus, since $W^{1,1}(0,T;\ZZ) \embeds C(0,T;\ZZ)$ and $D_z\II$ is continuous, 
a differential solutions fulfills $0 \in \partial\RR(0) + D_z\II(t,z(t))$ for all $t \in [0,T]$. 
The set $\mathcal{S}(t) := \{ z \in \ZZ \, : \,  0 \in \partial\RR(0) + D_z\II(t,z) \}$ is often called \emph{set of local stability}. 
Accordingly, a state $ z \in \mathcal{S}(t) $ is called \emph{locally stable}. 
The notion of a differential solution plays a crucial role in our error analysis. In case of a (globally) uniformly convex energy, 
one can prove that such a solution exists and is unique, see~Appendix~\ref{sec:diffexistence}.

As indicated above, there exists multiple other notions of solutions for \eqref{eq:subDiffInc}, among them 
\emph{(global) energetic solutions} and \emph{parameterized solutions}. These two solution concepts will appear in context of our 
numerical examples. They come into play, when one drops the uniform convexity assumption on the energy.
In the non-convex case, 
both solution concepts are especially essential in the context of incremental minimization time stepping schemes, 
as (weak) limits of the sequence of iterates are precisely of this type. 
To be more precise, weak accumulation points of the local scheme in \eqref{eq:locminscheme1} for $\tau \searrow 0$ 
are parameterized solutions, whereas weak accumulation points of its global counterpart 
(where the additional inequality constraint in \eqref{eq:locmin1} is dropped and the time update is just $t_{k+1} = t_k + \tau$)
are global energetic solutions.
For a precise definition of these two solution concepts and the convergence analysis in case of non-convex energies, 
we refer to \cite{knees17} and the references therein. 
Since only differential solutions will appear in our a priori analysis, we do not go into further details concerning 
the other notions of solutions.

\subsection{Local Minimization Algorithm}
In \citep{efenmielke06}, an implicit time stepping scheme based on a local minimization of dissipation plus energy was proposed to approximate parametrized solutions. This algorithm serves as a basis for our a priori analysis. Its iterates are determined by
\begin{subequations}\label{eq:locminscheme}
\begin{align} 
 z_{k} &\in \argmin\{\II(t_{k-1},z) + \RR(z-z_{k-1}) \, : \, z \in \ZZ, \, \norm{z-z_{k-1}}_\VV \leq \tau\}  \label{eq:locmin} \\
 t_{k} &= \min\{t_{k-1}+\tau-\norm{z_{k}-z_{k-1}}_\VV,T\} . \label{eq:tupdate} 
\end{align}
\end{subequations}
Note that the iterates implicitly depend on the choice of $\tau$. Nevertheless, we will omit any indexing of $t_k$ and $z_k$ for the sake of better readibility.
Now, for every $\tau>0$, we know from \citep{knees17} that this algorithm reaches the final time $T$ in a finite number of iterations (depending on $\tau$) which we will denote by $N(\tau)$. 
Moreover, by definition of $z_k$ as a solution of \eqref{eq:locmin}, it satisfies the 
necessary optimality conditions
\begin{equation}\label{eq:eq.aux004}
0 \in \partial (\RR+I_\tau)(z_{k}-z_{k-1}) + D_z\II(t_{k-1}, z_k),
\end{equation}
where $I_\tau: \VV \to [0,\infty]$ denotes the 
indicator functional associated with the constraint $v\in \overline{B_{\VV}(0,\tau)}$.
From \eqref{eq:eq.aux004}, we obtain the following optimality system:

\begin{lemma}[Discrete optimality System]\label{prop.optimalityprops} 
Let $k\geq 1$ and $\zth{k}$ be an arbitrary solution of \eqref{eq:locmin} 
with associated $\tth{k}$ given by \eqref{eq:tupdate}. 
Then the following optimality properties are satisfied:
There exists a Lagrange multiplier $\lambda_k \geq 0$ such that
\begin{subequations}
\begin{gather}
 \lambda_k(\norm{\zth{k}-\zth{k-1}}_\VV-\tau) = 0 \label{eq:opt.prop01} \\
 \tau \dist_{\VV^\ast}\{-D_z\II(\tth{k-1},\zth{k}),\partial \RR(0)\} 
 = \lambda_k \norm{\zth{k}-\zth{k-1}}_\VV^2 \label{eq:opt.prop02} \\
 \left\{\quad
 \begin{aligned}
  \RR(\zth{k}-\zth{k-1}) + \tau \dist_{\VV^\ast}\{-D_z\II(\tth{k-1},\zth{k}),\partial \RR(0)\}
  \qquad\quad & \\
  = \dual{-D_z\II(\tth{k-1},\zth{k})}{\zth{k}-\zth{k-1}}_{\ZZ^\ast,\ZZ} &
 \end{aligned} \right. \label{eq:opt.prop03} \\
 \RR(v) \geq - \dual{\lambda_k J_\VV(\zth{k}-\zth{k-1})
 + D_z\II(\tth{k-1},\zth{k})}{v}_{\VV^\ast,\VV} \quad\forall v\, \in \VV. \label{eq:opt.prop04} 
\end{gather}
\end{subequations}
\end{lemma}

For a proof of this statement, see \cite{knees17} or \cite{fem_paramsol}. 
Note that \eqref{eq:opt.prop02}--\eqref{eq:opt.prop04} and the 1-homogeneity of $\RR$ imply
\begin{equation}\label{eq:discrOptCon_DiffInc}
0 \in \partial\RR(z_k-z_{k-1}) + \lambda_k J_\VV (z_k-z_{k-1}) + D_z\II(t_{k-1},z_k) .
\end{equation}
In addition, \eqref{eq:opt.prop01} and \eqref{eq:opt.prop02} give
\begin{equation}\label{eq:lambda_dist_expression}
\lambda_k = \frac{1}{\tau} \dist_{\VV^\ast}\{-D_z\II(\tth{k-1},\zth{k}),\partial \RR(0)\} .
\end{equation}

\begin{remark}\label{rem:hatN}
In order to keep the following arguments concise, we will proceed the iteration for $t_{N(\tau)} = T$, until we find $z_{N(\tau)+n} \in \ZZ$, which is locally stable again, i.e., $0 \in \partial\RR(0) + D_z\II(t_{N(\tau)},z_{N(\tau)+n})$.
In \cref{lem:maxNum_noSteps} and \cref{lem:timeprogress} below, we will see that, under suitable assumptions, this condition is fulfilled after a finite number of steps, which is bounded independent of $\tau$. Eventually we denote $\hat{N}(\tau) := N(\tau)+n$.
\end{remark}

\begin{remark}\label{rem:z1=z0}	
Due to the convexity of $\II(t,\cdot)$ and the assumption on the initial state $z_0$, i.e., $0 \in \partial\RR(0) + D_z\II(t_0,z_0)$, 
there holds $\II(0,z_0) \leq \II(0,z) + \RR(z-z_0)$ for all $z \in \ZZ$
so that $z_1 = z_0$ is the unique minimizer of \eqref{eq:locmin}
and consequently, the first iterate of the local minimization algorithm always equals the initial state. 
This also entails $t_1 - t_0 = \tau$. We will use this fact at some places of the paper.
Note that the uniform convexity of  $\II(t_0,\cdot)$ on $B_\ZZ(z_0,\tau)$ is perfectly sufficient for the above argument, which will become important 
in Section~\ref{sec:LocConv} below.
\end{remark}

\subsection{A Counterexample in Case of a Non-Convex Energy}

Before we continue our error analysis, let us take a look at a first numerical example for the local minimization algorithm, which illustrates that on cannot expect any convergence result going beyond \citep{knees17,fem_paramsol} without further assumptions. For this example, we set $\ZZ=\VV=\XX=\R$ as well as:
\begin{align}
\RR(v) = \abs{v} \quad \text{and} \quad \II(t,z) = \frac{1}{2} z^2 + \FF(z) - \ell(t)z \label{eq:Energ_locConv_1D_1}
\end{align}
with
\begin{align*}
&\FF(z) = \left\{ \begin{array}{rc}
2z^3 - 5/2\, z^2 +1 & , z \geq 0 \\
-2z^3 - 5/2 \, z^2 +1 & , z<0
\end{array} \right. \quad \text{and} \quad  \ell(t) = -24(t-1/4)^2+5/3  .
\end{align*}
The fact that the energy functional is not (strictly) convex induces that solutions are in general not unique. However, it is a priori not clear, whether the discrete approximations converge to some particular parameterized solution (potentially even with some rate) or not. The following example demonstrates that this is in general \emph{not} the case. For $z_0 = -1/3$ straight forward calculations show that 
\[ z_1(t) \equiv -1/3 \quad \text{and} \quad 
 z_2(t) = \left\{
\begin{array}{rl}
-1/3 & ,t  \in [0,1/4) \\
1/3(1+\sqrt{2})  & ,t \in [1/4,1/2] 
\end{array} \right. \]
are solutions of the rate-independent system \eqref{eq:Energ_locConv_1D_1}. The numerical results depicted in Figure~\ref{fig:nonUnique_1D} show that, although $z_2$ is continuous, the discrete solution either approximates $z_1$ or $z_2$ depending on the choice of $\tau$. Consequently, as indicated above, without any form of (uniform) convexity of the energy-functional, it is not clear, if any of the solutions is preferred by the algorithm. In addition an a priori error estimate can hardly be expected. As a consequence of this example, we will impose additional assumptions on the energy to derive a priori error estimates. First we will assume that the energy is uniformly convex (Sections~\ref{sec:GlobConvLip} \& \ref{sec:GlobConvGen}) and later on generalize our results for the case of locally uniformly convex energies (Section~\ref{sec:LocConv}). 
\begin{figure}[h!]
  \centering   
  \begin{overpic}[scale=0.5]{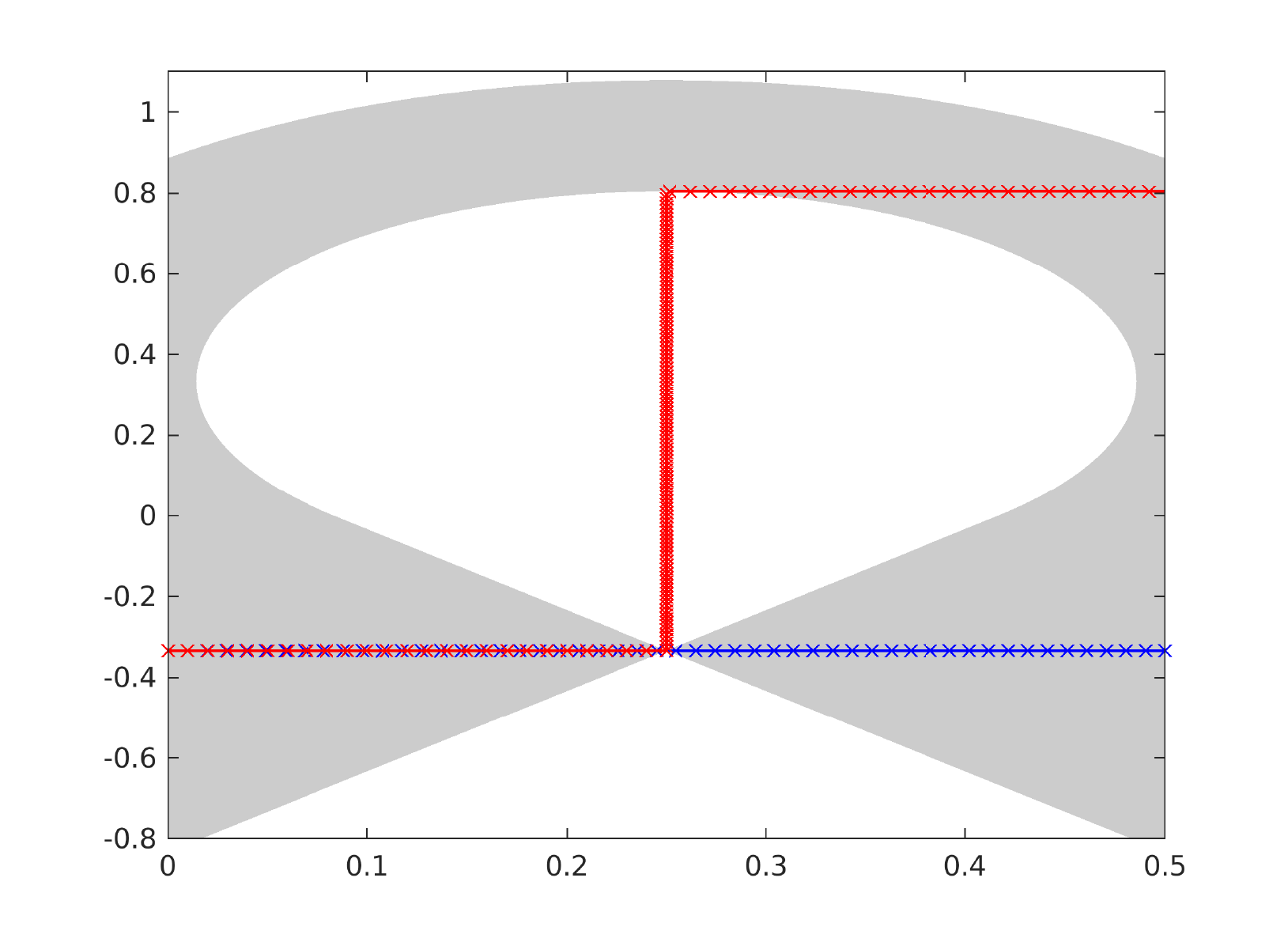}
     \put(69.4,29.75){\frame{\includegraphics[trim={140 100 120 100},clip,scale=0.41]{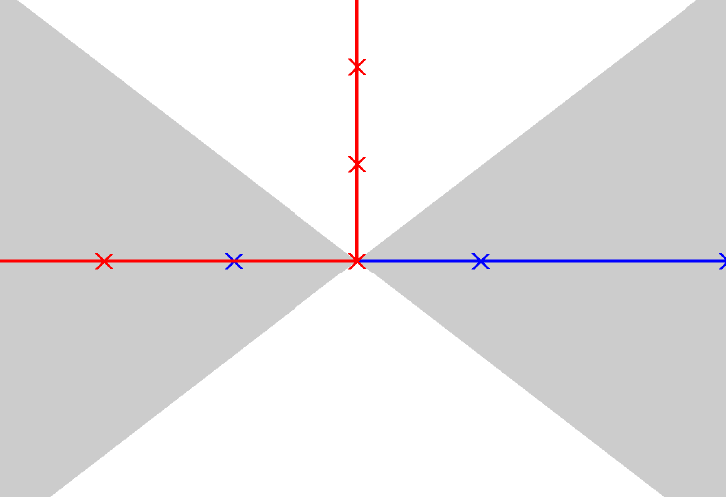}}}  
     \put(49,21){\frame{\makebox(5,3){\,}}}
     \put(49,24){\line(2,3){20.38}}
     \put(54,21){\line(6,1){51.76}}
  \end{overpic}
\caption{Approximations of two different parameterized solutions. Depending on the choice of $\tau$ either of two solutions is approximated. The set of local stability, i.e. $\cup_{ t \in [0,0.5] } \mathcal{S}(t) $, is depicted in gray.} 
     \label{fig:nonUnique_1D} 
\end{figure}


\section{A Priori Error Estimates}\label{sec:AprioriEstimates}

As mentioned above, the first part of our error analysis is based on the following

\begin{assumption}[$\kappa$-uniform convexity]\label{ass:globconv}\\
We say that $\II$ is \emph{$\kappa$-uniformly convex}, if there exists a $\kappa > 0$ such that, for all $t \in [0,T]$ and all $z,v \in \ZZ$, 
it holds $\dual{D^2_z\II(t,z)v}{v}_{\ZZ^*,\ZZ} \geq \kappa \norm{v}_\ZZ^2$. 
\end{assumption}

It is to be noted that, due to the structure of $\II$, the $\kappa$-uniform convexity is not depending on the time. 
Thus it suffices to require that $ z \mapsto \dual{Az}{z}+\FF(z)$ is $\kappa$-uniform convex. 
This property especially implies that 
\[ \dual{D_z\II(t,z_2)-D_z\II(t,z_1)}{z_2-z_1}_{\ZZ^*,\ZZ} \geq \kappa \norm{z_2-z_1}_\ZZ^2 \qquad \forall z_1,z_2 \in \ZZ . \] 
Later on, in Section~\ref{sec:LocConv}, we will relax this assumption and turn to locally uniformly convex energies, 
see Assumption~\ref{ass:locconv} below.

Before we start with our error analysis, we derive several auxiliary results that are frequently used throughout the whole paper.

\subsection{Basic Estimates}

In this section, we provide some basic estimates which will be useful for the error analysis in the upcoming subsections.

\begin{lemma}[Uniform a-priori estimate for iterates] \label{lem.basicest3}
The iterates of Algorithm~\ref{eq:locminscheme} fulfill $\sup_{\tau>0,\, k\in\N} \norm{\zth{k}}_\ZZ < \infty$. 
\end{lemma}

\begin{proof}
see \citep{knees17} or \cite{fem_paramsol}.
\end{proof}
Thus, we have that $z_k \in B_{\ZZ}(0,r_0)$ for all $k \in \N$ for some $r_0$ independent of $\tau$. The next result is essential in the context of parameterized solutions, since it implies that the artificial time is bounded and that the final time $T$ is reached within a finite number of iteration.
\begin{proposition}[Bound on artificial time] \label{prop.finiteNumberSteps}
For every $\tau>0$, there exists an index $N(\tau) \in \N$ such that $t_{N(\tau)} = T$. Moreover, it holds
$\sum_{k=1}^{N(\tau)} \norm{z_k-z_{k-1}}_\ZZ \leq C_{\Sigma}$
for some $C_{\Sigma}=C_{\Sigma}(\alpha,\FF,|\ell|_{Lip},z_0,T)>0$ independent of $\tau$.
\end{proposition}

\begin{proof}
see \citep{knees17} or \cite{fem_paramsol}.
\end{proof}

In what follows, we denote by $N(\tau)$ the number of necessary iterates to reach the final time at fineness $\tau$.
Finally, we state the following three auxiliary results, which will be used several times throughout this paper.
\begin{lemma}\label{lem:ehrling}
There exists $C_{\FF,r_0}>0$, such that for all $z_1,z_2 \in B_\ZZ(0,r_0)$:
\begin{equation*}
\abs{\dual{D_z\FF(z_1)-D_z\FF(z_2)}{v-w}_{\VV^*, \VV}} 
\leq C_{\FF,r_0} \, \norm{z_1-z_2}_\ZZ \, \norm{v-w}_\VV
\end{equation*}
for all $v,w \in \ZZ$.
\end{lemma}
\begin{proof}
The proof is a direct consequence of the growth-condition on $D^2_z\FF$. Let $v,w \in \ZZ$ be arbitrary. Using the aforementioned growth condition in \eqref{eq:assuF} together with the embedding $\ZZ \embeds \VV$ yields
\begin{align*}
\abs{\dual{D_z\FF(z_1)-D_z\FF(z_2)}{v-w}_{\VV^*, \VV}} 
 \leq C(1+r_0^q)\, \norm{z_1-z_2}_\ZZ \, \norm{v-w}_\VV .
\end{align*}
\end{proof}
\noindent
\begin{remark}\label{rem:ehrling2}
Thanks to Lemma~\ref{lem.basicest3} and \ref{lem:ehrling}, there is a constant $C_\FF>0$ such that, for all iterates $z_k,z_j \in \ZZ$ it holds
\begin{align*}
\abs{\dual{D_z\FF(z_1)-D_z\FF(z_2)}{v-w}_{\VV^*, \VV}} 
 \leq C(1+r_0^q)\, \norm{z_1-z_2}_\ZZ \, \norm{v-w}_\ZZ .
\end{align*}
\end{remark}

\begin{lemma}\label{lem:general_estimate}
Under the \cref{ass:globconv} we have for all iterates $k \in \N, \, k \leq N(\tau)$:
\begin{align}
0 &\geq \kappa \norm{z_{k+1}-z_{k}}_\ZZ^2 - |\ell|_{Lip} (t_{k}-t_{k-1}) \norm{z_{k+1}-z_{k}}_\VV + (\lambda_{k+1} - \lambda_k) \tau^2 . \label{eq:eq.aux002}
\end{align}
\end{lemma}
 \begin{proof}
First of all, we observe that, due to the complementarity condition in \eqref{eq:opt.prop01}, it holds 
$\lambda_k \norm{z_k-z_{k-1}}_\VV^2 = \lambda_k \tau^2$. 
Now, by inserting \eqref{eq:opt.prop02} in \eqref{eq:opt.prop03} and writing the resulting equation for the iteration $k+1$, we obtain
\begin{align}
\RR(z_{k+1}-z_{k}) = \dual{-D_z\II(t_{k},z_{k+1})}{z_{k+1}-z_{k}}_{\ZZ^\ast,\ZZ} - \lambda_{k+1} \tau^2 . \label{eq:eq.aux001}
\end{align}
Testing the inequality \eqref{eq:opt.prop04} with $v = z_{k+1}-z_k$ yields
\begin{align*}
\RR(z_{k+1}-z_{k}) 
&\geq \dual{-D_z\II(t_{k-1},z_{k})}{z_{k+1}-z_{k}}_{\ZZ^\ast,\ZZ} - \lambda_{k} \norm{z_k-z_{k-1}}_\VV \norm{z_{k+1}-z_k}_\VV \\
&\geq \dual{-D_z\II(t_{k-1},z_{k})}{z_{k+1}-z_{k}}_{\ZZ^\ast,\ZZ} - \lambda_{k}\tau^2
\end{align*}
Subtracting hereof the terms in \eqref{eq:eq.aux001}, exploiting the $\kappa$-uniform convexity of $\II(t,\cdot)$ and the Lipschitz-continuity of $\ell$, we obtain
\begin{align}
0
& \geq \dual{D_z\II(t_{k},z_{k+1})-D_z\II(t_{k},z_{k})}{z_{k+1}-z_{k}}_{\ZZ^\ast,\ZZ} \notag \\
&\quad + \dual{D_z\II(t_{k},z_{k})-D_z\II(t_{k-1},z_{k})}{z_{k+1}-z_{k}}_{\ZZ^\ast,\ZZ}  + (\lambda_{k+1}-\lambda_k) \tau^2 \notag \\
&\geq \kappa \norm{z_{k+1}-z_{k}}_\ZZ^2 - |\ell|_{Lip} (t_{k}-t_{k-1}) \norm{z_{k+1}-z_{k}}_\VV + (\lambda_{k+1} - \lambda_k) \tau^2 \label{eq:eq.aux005}
\end{align}
which was claimed.
\end{proof}
\begin{remark}
Revisiting the proof of \cref{lem:general_estimate}, we only needed the $\kappa$-uniform convexity in the last estimate. Since this will become important in the local uniform convex case, we state this estimate explicitly here: 
For all $k \in \N,$ $k \leq N(\tau)$, it holds (without assuming $\II$ to be uniformly convex):
\begin{align}
0 &\geq \dual{D_z\II(t_{k},z_{k+1})-D_z\II(t_{k},z_{k})}{z_{k+1}-z_{k}}_{\ZZ^\ast,\ZZ} \notag \\
&\quad + \dual{D_z\II(t_{k},z_{k})-D_z\II(t_{k-1},z_{k})}{z_{k+1}-z_{k}}_{\ZZ^\ast,\ZZ}  + (\lambda_{k+1}-\lambda_k) \tau^2 . \label{eq:eq.general_estimate_2} 
\end{align}
\end{remark}

\begin{lemma}\label{lem:estimate_consec_step}
Under \cref{ass:globconv} it holds for any $k \in \N$ with $k \leq N(\tau)$:
\begin{equation*} 
0 \in \partial\RR(0) + D_z\II(t_{k-1},z_k)
\quad \Longrightarrow \quad  \norm{z_{k+1}-z_k}_\ZZ \leq \frac{|\ell|_{Lip}}{\kappa} \, (t_k-t_{k-1}) 
\end{equation*}
\end{lemma}
\begin{proof} Let $0 \in \partial \RR(0) + D_z\II(t_{k-1},z_k)$,
which directly implies that $\lambda_k = 0$, due to \eqref{eq:lambda_dist_expression}. Thanks to \cref{lem:general_estimate} 
and the non-negativity of $\lambda_{k+1}$, we thus arrive at
$\kappa \norm{z_{k+1}-z_{k}}_\ZZ^2 - |\ell|_{Lip} (t_{k}-t_{k-1}) \norm{z_{k+1}-z_{k}}_\ZZ \leq 0$, 
where we used the embedding $\ZZ \embeds \VV$ with constant $ c_\ZZ = 1 $.
\end{proof}


\subsection{Globally Uniformly Convex Energy (in case $|\ell|_{Lip}$ is small)}\label{sec:GlobConvLip}

We are now in the position to start our error analysis. We begin with the case of a uniformly convex energy, 
see Assumption~\ref{ass:globconv}. Beside this, we additionally assume

\begin{assumption}[Bound on the Lipschitz constant of the driving force]\hfill\label{ass:Lipschitz_small}\\
There exists $\delta \in (0,\kappa]$ so that $|\ell |_{Lip} \leq \kappa - \delta$.
\end{assumption}

We will drop this assumption in the next subsection for the price of losing the optimal rate of convergence, 
see Theorem~\ref{thm:reverse_approx_general} below.

In order to define a discrete solution, we first introduce suitable interpolants in the artificial time: \\
For $s \in [s_{k-1},s_k) \subset [0,S_\ind)$, the continuous and piecewise affine interpolants are defined through
\begin{equation}\label{eq:affine-interpolants}
\hat{z}_\ind(s) := \zth{k-1} + \frac{(s-s_{k-1})}{s_k-s_{k-1}} ( \zth{k} - \zth{k-1}), \quad
\hat{t}_\ind(s) := \tth{k-1} + \frac{(s-s_{k-1})}{s_k - s_{k-1}} ( \tth{k} - \tth{k-1}), 
\end{equation}
while the piecewise constant interpolants are given by
\begin{equation*}
 \ovbar{z}_\ind(s) := \zth{k}, \qquad \ovbar{t}_\ind(s) := \tth{k}, \qquad 
 \ubar{z}_\ind(s) := \zth{k-1}, \qquad \ubar{t}_\ind(s) := \tth{k-1}.
\end{equation*}
The basic idea of our convergence proof is to transform the affine-interpolant back into the physical time and then to compare it with the unique differential solution of the rate-independent system \eqref{eq:subDiffInc}, which exists due to \citep[Thm 7.4]{mielketheil}. In order to guarantee that the back-transformation exists and fulfills some upper bounds, we need the following Lemma:

\begin{lemma}\label{lem:timeprogress}
Let \cref{ass:Lipschitz_small} hold. Then it holds that
\begin{equation}
\norm{z_{k+1} - z_{k}}_\ZZ \leq \frac{\kappa-\delta}{\kappa}\, (t_{k}-t_{k-1}) \quad \forall \, 1 \leq k \leq N(\tau)
\label{eq:timeprogress}
\end{equation} 
and $ (1- \frac{\kappa-\delta}{\kappa})= \frac{\delta}{\kappa} \leq \hat{t}_\tau^\prime(s) \leq 1 $ for almost all $s \in [0,\hat{S}_\tau]$. Moreover it holds $\hat{N}(\tau) = N(\tau)+1$.
\end{lemma}
\begin{proof} We argue by induction. 
Since $z_1 = z_0$ by Remark~\ref{rem:z1=z0}, 
we have $\partial \RR(z_1 - z_0) + D_z\II(t_0,z_1) = \partial \RR(0) + D_z\II(t_0,z_0) \ni 0$ so that \cref{lem:estimate_consec_step}
and Assumption~\ref{ass:Lipschitz_small} imply
\[  \norm{z_{2}-z_1}_\ZZ \leq \frac{|\ell|_{Lip}}{\kappa} \,  (t_1-t_0) \leq \frac{\kappa-\delta}{\kappa}\, (t_1-t_0), \]
which is \eqref{eq:timeprogress} for $k=1$.
Now, let $k\geq2$ be arbitrary and assume that \eqref{eq:timeprogress} holds for $k-1$, 
i.e., $\norm{z_{k} - z_{k-1}}_\ZZ \leq \frac{\kappa-\delta}{\kappa}\, (t_{k-1}-t_{k-2})<\tau$. 
Consequently, the complementarity conditions in \eqref{eq:opt.prop01} and \eqref{eq:discrOptCon_DiffInc} imply
\begin{equation*}
0 \in \partial \RR(z_k - z_{k-1}) + D_z\II(t_{k-1},z_k) \subset \partial \RR(0) + D_z\II(t_{k-1},z_k). 
\end{equation*}
Thus, by applying again \cref{lem:estimate_consec_step} and Assumption~\ref{ass:Lipschitz_small}, we obtain \eqref{eq:timeprogress} for the next iteration. 

For $s \in (0,\tau)$, the lower bound on $\hat{t}^\prime(s)$ follows immediately from $t_1 - t_0 = \tau$, see Remark~\ref{rem:z1=z0}.
For $s > \tau$, it is a direct consequence of \eqref{eq:timeprogress}, the embedding $\ZZ \embeds \VV$, and the time-update formula \eqref{eq:tupdate}. Finally, by \eqref{eq:timeprogress} and the complementarity condition \eqref{eq:opt.prop01}, we have $\lambda_{N(\tau)+1} = 0$, so that indeed $\hat{N}(\tau) = N(\tau)+1$ thanks to \eqref{eq:opt.prop02}.  
\end{proof}

We are now in the position to proof our main result on the convergence rate for parameterized solutions. By the Lemma above, there exists an unique inverse function $\hat{s}_\tau(t) : [0,T] \mapsto [0,\hat{S}_\tau]$ of $\hat{t}_\tau$. 
We will then denote by $z_\tau(t) := \hat{z}_\tau(s_\tau(t)) $ the retransformed discrete parameterized solution (see also end of the proof of \cref{thm:reverse_approx_small}).


\begin{theorem}\label{thm:reverse_approx_small}
Let $\II(t,\cdot) \in C^{2,1}_{loc}(\ZZ;\R)$ (see \eqref{ass:regularity_I}) as well as \cref{ass:globconv} and \cref{ass:Lipschitz_small} hold. Moreover let $\ell \in W^{1,\infty}([0,T];\VV^\ast)$ with $\ell^\prime \in BV([0,T];\VV)$. Then, the sequence $\{z_\tau\}_{\tau>0}$ of retransformed discrete parameterized solutions converges to the unique (differential) solution $z$ and satisfies the a-priori error estimate
\begin{equation}\label{eq:reverse_approx_small}
\norm{z_\tau(t) - z(t)}_\ZZ \leq K \, \tau \qquad \forall t \in [0,T],
\end{equation}
where $K=K(\alpha,\kappa,\ell,z_0,T,\FF,\norm{A}_{\mathcal{L}(\ZZ,\ZZ^\ast)})>0$ is independent of $\tau$.
\end{theorem}
\begin{proof}For convenience of the reader we split the rather lengthy proof into eight parts, which are as follows:
\begin{enumerate}
    \item[0.] First, we will see that, due to the uniform convexity of the energy, \eqref{eq:subDiffInc} even admits a unique differential solution and not only 
    a parameterized one.
    \item[1.] Based on Lemma~\ref{lem:timeprogress}, we can transform the piecewise affine interpolants introduced above to the physical time. 
    This allows to compare the discrete solution with the exact (differential) solution, which of course also ``lives'' in the physical time.
    The error analysis however uses a slightly different piecewise affine interpolant, denoted by $\tilde z_\tau$ providing a certain shift in the time steps.
    \item[2.] In analogy to \cite{mielketheil}, we introduce a quantity $\gamma(t)$, which dominates the pointwise error $\|\tilde z_\tau (t) - z(t)\|_\ZZ$.
    This error measure enables us to deal with uniformly convex energy functionals instead of just quadratic and coercive ones.
    \item[3.] The error measure is essentially estimated by two contributions, denoted by $E(t)$ and $R(t)$. Both contributions depend only 
    differences of $D_z\II$ evaluated at different time points and different discrete solutions.
    \item[4.~\&~5.] $E(t)$ and $R(t)$ are estimated by using the smoothness properties of $\FF$ and the load $\ell$. In addition, 
    the uniform convexity of $\II$ plays an essential role for the estimate of $R$. In this way, one obtains a estimate of $\Landau{\tau^2}$ for the 
    $L^1$-norms of $E$ and $R$.
    \item[6.] Together with Gronwall's lemma, this estimate yields a bound of $\Landau{\tau}$ for the error indicator $\gamma$ and thus also for the 
    error $\|\tilde z_\tau (t) - z(t)\|_\ZZ$. 
    \item[7.] Finally, we relate the $\|\tilde z_\tau (t) - z(t)\|_\ZZ$ with the auxiliary interpolant $\tilde z_\tau$ to the ``true error'' containing the 
    ``correct'' interpolant $z_\tau = \hat z_\tau \circ s_\tau$ as introduced above.
\end{enumerate}
\vspace*{1ex}

\noindent
\textsl{Step 0: Differential Solution}\\
First of all, due to \cref{thm:app.existence_and_uniqueness}, there exists a unique (differential) solution $z \in C^{0,1}(0,T;\ZZ)$ of the rate-independent system. In particular, it holds f.a.a.~$t\in [0,T]$ that $0 \in \partial\RR(z^\prime(t)) + D_z\II(t,z(t))$,
which can be reformulated as (see \eqref{eq:app.uniqueDiffSol_reformulations}):
\begin{subequations}
\begin{align}
\forall v \in \ZZ: \qquad \RR(v) &\geq \dual{-D_z\II(t,z(t))}{v}_{\ZZ^\ast,\ZZ} \quad  &\forall t \in [0,T] \label{eq:uniqueDiffSol_reformulation1} \\
\RR(z^\prime(t)) &= \dual{-D_z\II(t,z(t))}{z^\prime(t)}_{\ZZ^\ast,\ZZ} \quad &\faa t \in [0,T]. \label{eq:uniqueDiffSol_reformulation2}
\end{align}
\end{subequations}
Since $ z \in C^{0,1}(0,T;\ZZ)$, it additionally holds 
\begin{equation}\label{eq:bound_diff_sol_thm1}
\norm{z^\prime(t)}_\ZZ \leq C \quad \faa t \in [0,T].
\end{equation}

\noindent
\textsl{Step 1: Construction of interpolants in the physical time}\\
Given $t \in [t_{k-1}^\tau,t_k^\tau)$ with $k\leq N(\tau)$, we define the following affine interpolant
\begin{gather}\label{eq:affineInterpol_thm1}
\tilde{z}_\tau(t) = z_{k}^\tau + \frac{t-t_{k-1}^\tau}{t_k^\tau-t_{k-1}^\tau}(z_{k+1}^\tau - z_{k}^\tau).
\end{gather} 
Note that $[t_{k-1},t_k)$ is nonempty and that $\lambda_k = 0$ due to \cref{lem:timeprogress}. Thus, from the first order optimality condition for the local minimization problem, i.e. \eqref{eq:discrOptCon_DiffInc}, we know that $0 \in \partial\RR(\tilde{z}_\tau^\prime(t)) + D_z\II(t_{k},z_{k+1})$.
Analogous to \textsl{Step 0}, this can be reformulated as 
\begin{subequations}
\begin{align}
\forall v \in \ZZ: \qquad \RR(v) &\geq \dual{-D_z\II(t_{k},z_{k+1})}{v}_{\ZZ^\ast,\ZZ} \quad  &\forall k \in \{0,\dots,N(\tau)\}  \label{eq:discreteSol_reformulation1}\\
\RR(\tilde{z}_\tau^\prime(t)) &= \dual{-D_z\II(t_k,z_{k+1})}{\tilde{z}_\tau^\prime(t)}_{\ZZ^\ast,\ZZ} \quad &\faa t \in [0,T] \label{eq:discreteSol_reformulation2}
\end{align}
\end{subequations}
Exploiting \cref{lem:timeprogress}, we additionally have
\begin{equation}\label{eq:discr_deriv_bound_thm1}
\norm{\tilde{z}_\tau^\prime(t)}_\ZZ \leq C \quad \faa t \in [0,T].
\end{equation}

\noindent
\textsl{Step 2: Introduction of an error measure}\\
We now basically follow the lines of \citep[Thm 7.4]{mielketheil}, but have to adapt the underlying analysis at some points. Therefore we present the arguments in detail. Let us define 
\begin{equation}\label{eq:defgamma}
 \gamma(t) := \dual{D_z\II(t,\tilde{z}_\tau(t)) - D_z\II(t,z(t))}{\tilde{z}_\tau(t)-z(t)}_{\ZZ^\ast,\ZZ} \,.
\end{equation}
Due to the $\kappa$-uniform convexity of $\II(t,\cdot)$, we have 
\begin{equation}\label{eq:gammaest}
    \gamma(t) \geq \kappa \norm{\tilde{z}_\tau(t)-z(t)}_\ZZ^2
\end{equation}
so that $\gamma$ measures the discretization error. In full analogy to \citep[Thm 7.4]{mielketheil}, we can estimate (see Appendix~\ref{sec:App.errorGamma})
\begin{equation}
\dot{\gamma}(t) \leq C \, \norm{\tilde{z}_\tau(t)-z(t)}_\ZZ^2 
+ 2 \dual{D_z\II(t,\tilde{z}_\tau(t)) - D_z\II(t,z(t))}{\tilde{z}^\prime_\tau(t)-z^\prime(t)}_{\ZZ^\ast,\ZZ} \, , \label{eq:ddt-gamma_thm1}
\end{equation}
for almost all $t \in [0,T]$. We split the second term into two parts, namely
\begin{align*}
e_1(t) &:= 2 \, \dual{D_z\II(t,z(t)) - D_z\II(t,\tilde{z}_\tau(t))}{z^\prime(t)}_{\ZZ^\ast,\ZZ} \\
\text{and} \quad e_2(t) &:= 2 \, \dual{D_z\II(t,\tilde{z}_\tau(t)) - D_z\II(t,z(t))}{\tilde{z}^\prime_\tau(t)}_{\ZZ^\ast,\ZZ} \, . 
\end{align*}
\textsl{Step 3: Estimates for the error $e_i$}\\
Let again $k \leq N(\tau)$ and $t \in [t_{k-1},t_k)$ be arbitrary. First observe that, due to the convexity of $\partial \RR(0)$, it holds for 
\[ \theta(t) = \frac{t-t_{k-1}}{t_k-t_{k-1}} \]
that $-(1-\theta(t)) \, \xi_{k-1} - \theta(t) \, \xi_k \in \partial \RR(0)$ with 
$\xi_{k-1} := D_z\II(t_{k-1},z_k)$ and $\xi_k := D_z\II(t_k,z_{k+1})$.
From the characterization of $\partial\RR(0)$, we infer 
$\RR(v) \geq - \dual{(1-\theta(t)) \, \xi_{k-1} + \theta(t) \, \xi_k}{v}_{\ZZ^\ast,\ZZ}$ for all $v \in \ZZ$.
Inserting herein $v = z^\prime(t)$ and substracting \eqref{eq:uniqueDiffSol_reformulation2}, we can estimate 
\begin{align}
e_1(t) &= 2 \, \dual{D_z\II(t,z(t)) - \left(1-\theta(t) \right)\xi_{k-1} - \theta(t) \, \xi_k}{z^\prime(t)}_{\ZZ^\ast,\ZZ} \notag \\
&\qquad + 2 \dual{\left(1-\theta(t) \right)\xi_{k-1} + \theta(t) \, \xi_k - D_z\II(t,\tilde{z}_\tau(t))}{z^\prime(t)}_{\ZZ^\ast,\ZZ} \notag \\
&\leq 2 \, \norm{\left(1-\theta(t) \right)\xi_{k-1} + \theta(t) \, \xi_k - D_z\II(t,\tilde{z}_\tau(t))}_{\ZZ^\ast} \norm{z^\prime(t)}_\ZZ \label{eq:e1_estimate_thm1}
\end{align}
for almost all $t \in [t_{k-1},t_k)$.

Next we turn to the term $e_2$. Similarly, we take $v = \tilde{z}_\tau^\prime(t)$ in \eqref{eq:uniqueDiffSol_reformulation1} and substract \eqref{eq:discreteSol_reformulation2} to obtain $0 \geq \dual{D_z\II(t_{k},z_{k+1}) - D_z\II(t,z(t))}{\tilde{z}_\tau^\prime(t)}_{\ZZ^\ast,\ZZ}$,
from which we deduce
\begin{equation}\label{eq:e2_estimate_thm1}
\begin{aligned}
e_2(t) 
&\leq 2 \, \dual{D_z\II(t,\tilde{z}_\tau(t)) - D_z\II(t_{k},z_{k+1})}{\tilde{z}^\prime_\tau(t)}_{\ZZ^\ast,\ZZ}\\
&\leq 2 \dual{D_z\II(t,\tilde{z}_\tau(t)) - \left(1-\theta(t) \right)\xi_{k-1} - \theta(t) \, \xi_k}{\tilde{z}_\tau^\prime(t)}_{\ZZ^\ast,\ZZ} \\
&\qquad + 2 \dual{\left(1-\theta(t) \right)\xi_{k-1} + \theta(t) \, \xi_k - D_z\II(t_{k},z_{k+1})}{\tilde{z}_\tau^\prime(t)}_{\ZZ^\ast,\ZZ} \\
&\leq 2 \, \norm{\left(1-\theta(t) \right)\xi_{k-1} + \theta(t) \, \xi_k - D_z\II(t,\tilde{z}_\tau(t))}_{\ZZ^\ast} \norm{\tilde{z}_\tau^\prime(t)}_\ZZ \notag \\
&\qquad + 2 (1-\theta(t)) \dual{D_z\II(t_{k-1},z_k)-D_z\II(t_{k},z_{k+1})}{\tilde{z}_\tau^\prime(t)}_{\ZZ^\ast,\ZZ}.
\end{aligned}
\end{equation}
Next, let us define
\begin{align} 
E(t) &:= \norm{\left(1-\theta(t) \right)\xi_{k-1} + \theta(t) \, \xi_k - D_z\II(t,\tilde{z}_\tau(t))}_{\ZZ^\ast} \label{eq:E_thm1}\\
\text{and } \quad R(t) &:= 2 (1-\theta(t)) \dual{D_z\II(t_{k-1},z_k)-D_z\II(t_{k},z_{k+1})}{\tilde{z}^\prime_\tau(t)}_{\ZZ^\ast,\ZZ} . \label{eq:R_thm1}
\end{align}
Then we insert \eqref{eq:E_thm1} and \eqref{eq:R_thm1} in \eqref{eq:e1_estimate_thm1} and \eqref{eq:e2_estimate_thm1}. The resulting estimates for $e_1$ and $e_2$ are in turn inserted in \eqref{eq:ddt-gamma_thm1}, which, together with the boundedness of $\norm{z^\prime(t)}_\ZZ$ and $\norm{\tilde{z}_\tau^\prime(t)}_\ZZ$ by \eqref{eq:bound_diff_sol_thm1} and \eqref{eq:discr_deriv_bound_thm1}, yields 
\begin{equation}
\dot{\gamma}(t) \leq C ( \norm{\tilde{z}_\tau(t) - z(t)}_\ZZ^2 + E(t) + R(t) ) . \label{eq:gamma-E-R_thm1}
\end{equation}

\noindent
\textsl{Step 4: Estimate for $E(t)$}\\
The particular structure of $\II$ together with the linearity of $A$ and the definition of $\tilde{z}_\tau$ gives
\begin{align*}
E(t) &\leq 
\norm{(1-\theta(t)) D_z\FF(z_{k}) + \theta(t) D_z\FF(z_{k+1}) - D_z\FF\left( (1-\theta(t))z_{k} - \theta(t) z_{k+1}\right)}_{\ZZ^\ast} \\
&\qquad + \norm{(1-\theta(t)) \ell(t_{k-1}) + \theta(t) \ell(t_k) - \ell(t)}_{\ZZ^\ast} \\
&=: I_1(t) + I_2(t).
\end{align*}
Exploiting the regularity of $\FF$, we can estimate
\begin{align*}
I_1(t)  
&\leq \theta(t) \norm{z_{k+1}-z_k}_\ZZ  \\
& \;\quad \int_0^1 \vnorm{ D_z^2\FF(z_k+s (z_{k+1}-z_k)) - D_z^2\FF(z_k+s \theta(t)(z_{k+1}-z_k))}_{\LL(\ZZ,\LL(\ZZ,\ZZ^\ast))} \d s \\
&\leq C \norm{z_{k+1}-z_k}_\ZZ^2 , 
\end{align*}
where we also used $ \theta(t) \in [0,1]$ and the boundedness of the iterates $z_{k}$ in $\ZZ$ independent of $\tau$ from \cref{lem.basicest3}. For $I_2$, we proceed similarly by exploiting the regularity of $\ell$: 
\begin{align*}
I_2(t) \leq 
\int_{t_{k-1}}^{t} \left\| \frac{\ell(t_k)-\ell(t_{k-1})}{t_k-t_{k-1}} - \ell^\prime(s) \right\|_{\VV} \d s 
\leq \tau \norm{\ell^{\prime}}_{BV([t_{k-1},t_k];\VV)} \, \d s . 
\end{align*}
Since $\norm{z_{k+1}-z_k}_\ZZ \leq C \tau$ by \cref{lem:timeprogress}, the above estimates for $I_1(t)$ and $I_2(t)$ imply 
for all $t \in [t_{k-1},t_k)$ that $E(t) \leq C \tau^2 + \tau \norm{\ell^{\prime}}_{BV([t_{k-1},t_k];\VV)}$. 
Now integrating $E$ yields
\begin{equation}
\int_0^T E(t) \d t \leq C \tau^2 + \tau^2 \norm{\ell^{\prime}}_{BV([0,T];\VV)} \leq C \tau^2 . \label{eq:E_estimate_thm1}
\end{equation} 

\noindent
\textsl{Step 5: Estimate for $R(t)$}\\
First, we abbreviate $\EE(z) := \dual{Az}{z} + \FF(z)$ so that $\II(t,z) = \EE(z) - \dual{\ell(t)}{z}$, as well as 
\begin{alignat*}{5}
\Delta t_{k} &:= t_k-t_{k-1}, &\quad \d_\tau \ell_k &:= \frac{\ell(t_k)-\ell(t_{k-1})}{\Delta t_k}, & \quad & k = 1, ..., N(\tau),\\
\d_\tau z_{k+1} &:= \frac{z_{k+1}-z_{k}}{\Delta t_k}, & \quad
\d_\tau D_z\EE_{k+1} &:= \frac{D_z\EE(z_{k+1})-D_z\EE(z_k) }{\Delta t_k},
& \quad & k = 1, ..., N(\tau) - 1,
\end{alignat*} 
as well as $\d_\tau \ell_0 = 0$, $\d_\tau z_1 = 0$, and $\d_\tau D_z\EE_1 = 0$.
By \cref{lem:timeprogress}, we have
\begin{equation}
\norm{ \d_\tau z_k}_\ZZ \leq C . \label{eq:dt_zk_bound}
\end{equation}
Now, on account of $-D_z\II(t_{k-1},z_{k}) \in \partial\RR(z_k-z_{k-1})$, we deduce from \eqref{eq:discreteSol_reformulation1} tested with $z_k-z_{k-1}$
that $0 \geq \dual{D_z\II(t_{k-1},z_k) - D_z\II(t_{k},z_{k+1})}{z_{k}-z_{k-1}}_{\ZZ^\ast,\ZZ}$.
Inserting the definitions of $\tilde{z}$ and $\theta(t)$, we thus obtain for $t \in [\tth{k-1},\tth{k})$:
\begin{align*}
R(t) 
&= 2 (t_k-t) \dual{(\Delta t_k)^{-1} [D_z\II(t_{k-1},z_k)-D_z\II(t_{k},z_{k+1})]}{\d_\tau z_{k+1}-\d_\tau z_k}_{\ZZ^\ast,\ZZ} \\
&\qquad  + 2 (t_k-t) \dual{(\Delta t_k)^{-1}[D_z\II(t_{k-1},z_k)-D_z\II(t_{k},z_{k+1})]}{\d_\tau z_k}_{\ZZ^\ast,\ZZ} \\
&\leq 2 (t_k-t) \dual{(\Delta t_k)^{-1} [D_z\II(t_{k-1},z_k)-D_z\II(t_{k},z_{k+1})]}{\d_\tau z_{k+1}-\d_\tau z_k}_{\ZZ^\ast,\ZZ} \\
&= 2 (t_k-t) \dual{-\d_\tau D_z\EE_{k+1} + \d_\tau \ell_k}{\d_\tau z_{k+1}-\d_\tau z_k}_{\ZZ^\ast,\ZZ} .
\end{align*}
Integrating then gives
\begin{align}
\int_0^T R(t) \d t 
&\leq \sum_{k=1}^{N(\tau)} (\Delta t_k)^2 \dual{-\d_\tau D_z\EE_{k+1} + \d_\tau \ell_k}{\d_\tau z_{k+1}-\d_\tau z_k}_{\ZZ^\ast,\ZZ} \notag \\
&\leq \tau^2 \sum_{k=1}^{N(\tau)} \dual{-\d_\tau D_z\EE_{k+1}}{\d_\tau z_{k+1}-\d_\tau z_k}_{\ZZ^\ast,\ZZ} 
 + \dual{\d_\tau \ell_k}{\d_\tau z_{k+1}-\d_\tau z_k}_{\ZZ^\ast,\ZZ} . \label{eq:R_estimate_thm1}
\end{align} 
For the terms involving $\ell$ we have
\begin{align*}
&\sum_{k=1}^{N(\tau)} \dual{\d_\tau \ell_k}{\d_\tau z_{k+1}-\d_\tau z_k}_{\VV^\ast,\VV} \\
&\qquad = \sum_{k=1}^{N(\tau)} \dual{\d_\tau \ell_k}{\d_\tau z_{k+1}}_{\VV^\ast,\VV} - \dual{\d_\tau \ell_k - \d_\tau \ell_{k-1}}{\d_\tau z_k}_{\VV^\ast,\VV} - \dual{\d_\tau \ell_{k-1}}{\d_\tau z_k}_{\VV^\ast,\VV} , \\ 
\end{align*}
where we used $\d_\tau \ell_0 = 0$. 
The second term is estimated analogously to $I_2$, exploiting the regularity of $\ell$ as well as the boundedness of $\norm{\d_\tau z_k}_\VV$ from \eqref{eq:dt_zk_bound}, which yields 
\begin{align*}
&\hspace*{-0.5cm}\abs{\dual{\d_\tau \ell_k - \d_\tau \ell_{k-1}}{\d_\tau z_k}_{\VV^\ast,\VV}} \\
&=\Big|\int_0^1 \dual{\ell^\prime(t_{k-1}+s(t_k-t_{k-1})) - \ell^\prime(t_{k-2}+s(t_{k-1}-t_{k-2})) \d s}{\d_\tau z_k}_{\VV^\ast,\VV} \Big| \\
&\leq \norm{\ell^{\prime}}_{BV([t_{k-2},t_k];\VV)} \, \norm{\d_\tau z_k}_\VV \leq C \norm{\ell^{\prime}}_{BV([t_{k-2},t_k];\VV)} .
\end{align*}
Hence, thanks to $\d_\tau \ell_0 = 0$ and \eqref{eq:dt_zk_bound},
\begin{align}
&\sum_{k=1}^{N(\tau)} \dual{\d_\tau \ell_k}{d_t z_{k+1}-d_t z_k}_{\VV^\ast,\VV} \notag \\[-1ex]
&\leq \sum_{k=1}^{N(\tau)} \dual{\d_\tau \ell_k}{\d_\tau z_{k+1}}_{\VV^\ast,\VV} - \dual{\d_\tau \ell_{k-1}}{\d_\tau z_k}_{\VV^\ast,\VV} + C \norm{\ell^{\prime}}_{BV([t_{k-2},t_k];\VV)}  \notag \\
&\leq \dual{\d_\tau \ell_{N(\tau)-1}}{\d_\tau z_{N(\tau)}}_{\VV^\ast,\VV} + 2 C \norm{\ell^{\prime}}_{BV([0,T];\VV)} 
\leq C( \abs{\ell}_{Lip} +  \norm{\ell^{\prime}}_{BV([0,T];\VV)} ) . \label{eq:est-sum_dt-lk}
\end{align}
Now, for the terms involving $D_z\EE$, we first calculate
\begin{align*}
\dual{\d_\tau D_z\EE_{k+1}}{\d_\tau z_k}_{\ZZ^\ast,\ZZ}
&= \Big\langle \frac{D_z\EE(z_{k+1})-D_z\EE(z_{k})}{t_k-t_{k-1}} , \d_\tau z_k \Big\rangle_{\ZZ^\ast,\ZZ} \\
&= \int_0^1 \dual{D_z^2\EE(z_k+s(z_{k+1}-z_k))[\d_\tau z_{k+1}]}{\d_\tau z_k}_{\ZZ^\ast,\ZZ} \, \d s .
\end{align*}
Since $D_z^2\EE$ is symmetric, we obtain
\begin{align*}
&2 \int_0^1 \dual{D_z^2\EE(z_k+s(z_{k+1}-z_k))[\d_\tau z_{k+1}]}{\d_\tau z_k}_{\ZZ^\ast,\ZZ} \, \d s \\
&= - \int_0^1 \dual{D_z^2\EE(z_k+s(z_{k+1}-z_k))[\d_\tau z_{k+1}-\d_\tau z_k]}{\d_\tau z_{k+1}- \d_\tau z_k}_{\ZZ^\ast,\ZZ} \, \d s \\
&\quad+ \int_0^1 \dual{D_z^2\EE(z_k+s(z_{k+1}-z_k))[\d_\tau z_{k+1}]}{\d_\tau z_{k+1}}_{\ZZ^\ast,\ZZ} \, \d s \\
&\quad+ \int_0^1 \dual{(D_z^2\EE(z_k+s(z_{k+1}-z_k))- D_z^2\EE(z_{k-1}+s(z_{k}-z_{k-1}))) [\d_\tau z_{k}]}{\d_\tau z_k}_{\ZZ^\ast,\ZZ} \, \d s \\
&\quad+ \int_0^1 \dual{D_z^2\EE(z_{k-1}+s(z_{k}-z_{k-1}))[\d_\tau z_{k}]}{\d_\tau z_k}_{\ZZ^\ast,\ZZ} \, \d s
\end{align*}
Thus, thanks to the convexity of $\EE$, we have 
\begin{align*}
\dual{\d_\tau D_z\EE_{k+1}}{\d_\tau z_k}_{\ZZ^\ast,\ZZ} &\leq \frac{1}{2} \dual{\d_\tau D_z\EE_k}{\d_\tau z_k}_{\ZZ^\ast,\ZZ} + \frac{1}{2} \dual{\d_\tau D_z\EE_{k+1}}{\d_\tau z_{k+1}}_{\ZZ^\ast,\ZZ} \\
&\qquad + \frac{1}{2}C \norm{\d_\tau z_{k}}_\ZZ^2 (\norm{z_{k+1}-z_{k}}_\ZZ + \norm{z_k-z_{k-1}}_\ZZ ) ,
\end{align*}
where we also used the regularity of $\EE$. Exploiting Proposition ~\ref{prop.finiteNumberSteps} and \eqref{eq:dt_zk_bound}, we eventually end up with 
\begin{align*}
&\hspace*{-0.5cm}\sum_{k=1}^{N(\tau)} \dual{\d_\tau D_z\EE_{k+1}}{\d_\tau z_{k}}-\dual{\d_\tau D_z\EE_{k+1}}{\d_\tau z_{k+1}}_{\ZZ^\ast,\ZZ} \\
&\leq \frac{1}{2} \sum_{k=1}^{N(\tau)} \begin{aligned}[t] \{ &\dual{\d_\tau D_z\EE_{k}}{\d_\tau z_k}_{\ZZ^\ast,\ZZ} - \dual{\d_\tau D_z\EE_{k+1}}{\d_\tau z_{k+1}}_{\ZZ^\ast,\ZZ} \\
&+ C \norm{\d_\tau z_{k}}_\ZZ^2 (\norm{z_{k+1}-z_{k}}_\ZZ + \norm{z_k-z_{k-1}}_\ZZ ) \} \end{aligned} \\
&\leq C C_\Sigma + \frac{1}{2} \dual{\d_\tau D_z\EE_1}{\d_\tau z_1}_{\ZZ^\ast,\ZZ} - \frac{1}{2} \dual{\d_\tau D_z\EE_{N(\tau)+1}}{\d_\tau z_{N(\tau)+1}}_{\ZZ^\ast,\ZZ} \leq C .
\end{align*}
wherein the last estimate is due to Remark~\ref{rem:z1=z0}, i.e., $\dual{\d_\tau D_z\EE_1}{\d_\tau z_1} = 0$, and the convexity of $\EE$, that is $\dual{\d_\tau D_z\EE_{N(\tau)+1}}{\d_\tau z_{N(\tau)+1}} \geq 0$. Combining this with \eqref{eq:E_estimate_thm1}, \eqref{eq:est-sum_dt-lk} and \eqref{eq:R_estimate_thm1}, we have overall shown that 
\begin{equation}\label{eq:eq.EandR}
\int_0^T E(t) \d t + \int_0^T R(t) \d t \leq C \tau^2 .
\end{equation}
\textsl{Step 6: Obtain Convergence Rate by Gronwall Lemma}\\
Exploiting that $\gamma(t)/\kappa \geq \norm{\tilde{z}_\tau(t)-z(t)}_\ZZ^2$ in \eqref{eq:gamma-E-R_thm1}, one obtains
$\dot{\gamma}(t) \leq C ( \gamma(t) + E(t) + R(t) )$.
Integrating this and using Gronwall's Lemma as well as the estimates \eqref{eq:eq.EandR} on $E$ and $R$ yield
$\gamma(t) \leq (\gamma(0) + C \tau^2)\exp^{C t} \leq C (\gamma(0)+\tau^2)$.
Due to $\tilde{z}_\tau(0) = z(0) = z_0$, we have $\gamma(0) = 0$. 
Using another time the $\kappa$-uniform convexity of $\II$, we therefore finally obtain
\begin{equation}\label{eq:tildeerr}
    \norm{\tilde{z}_\tau(t)-z(t)}_\ZZ^2 \leq \gamma(t)/\kappa \leq C \tau^2 . 
\end{equation}
\textsl{Step 7: Comparing interpolants}\\
By $\hat{z}_\tau$ we denote the affine interpolation of the discrete approximations with stepsize $\tau$ in the artificial time, see \eqref{eq:affine-interpolants}. From \cref{lem:timeprogress}, we conclude that $\hat{t}_\tau(s)$ is monotonically increasing and $\hat{t}_\tau^\prime(s) \geq  1- \frac{\kappa-\delta}{\kappa} $ a.e. in $[0,\hat{S}_\tau]$. Thus, there exists a unique inverse function $s_\tau : [0,T] \to [0,\hat{S}_\tau]$ with $ 1 \leq s_\tau^\prime(t) \leq \frac{1}{1- \frac{\kappa-\delta}{\kappa}}$ a.e. in $[0,T]$. Given this inverse, one can define $\hat{z}_\tau$ as the retransformed affine interpolant, 
i.e., $z_\tau(t) := \hat{z}_\tau(s_\tau(t))$.
By elementary calculations, 
the explicit formula for $z_\tau$ is:
\begin{gather*}
z_\tau(t) = z_{k-1}^\tau + \frac{t-t_{k-1}^\tau}{t_k^\tau-t_{k-1}^\tau}(z_k^\tau - z_{k-1}^\tau), \quad t \in [t_{k-1}^\tau,t_k^\tau),
\end{gather*}
i.e., $z_\tau$ is just the affine interpolant in the physical time. 
Comparing $z_\tau$ with $\tilde{z}_\tau$ from \eqref{eq:affineInterpol_thm1} results in
\begin{multline*}
\norm{z_\tau(t)-\tilde{z}_\tau(t)}_\ZZ 
= \norm{z_{k-1}^\tau + \theta(t) (z_k^\tau - z_{k-1}^\tau) - z_{k}^\tau - \theta(t) (z_{k+1}^\tau - z_{k}^\tau)}_\ZZ \\
\leq (1-\theta(t)) \norm{z_{k-1}^\tau-z_k^\tau}_\ZZ + \theta(t) \norm{z_k^\tau-z_{k+1}^\tau}_\ZZ 
\leq \tau .
\end{multline*}
where we exploited \eqref{eq:timeprogress} once more. Now, since $k \leq N(\tau)$ was arbitrary, we have $\norm{z_\tau(t)-\tilde{z}_\tau(t)}_\ZZ \leq  \tau$ for all $t \in [0,T]$. In combination with \eqref{eq:tildeerr}, this finally gives
$\norm{z_\tau(t)-z(t)}_\ZZ \leq  K \tau$,
which is the desired result. A careful analysis of the used estimates and the corresponding constants yields that $K$ provides the claimed dependencies.
\end{proof}

Some remarks and comments concerning the assertion of Theorem~\ref{thm:reverse_approx_small} and its proof are in order:

\begin{remark}\label{rem:uni-conv_estimates}
In preparation of Section~\ref{sec:LocConv} below, we note that the uniform convexity of the energy is only needed at three places in the above
analysis: firstly for the estimate in \eqref{eq:timeprogress}, 
secondly for the lower bound on $\gamma$ in \eqref{eq:gammaest}, and thirdly for the inequality
\begin{equation}
 \int_0^1 \dual{D_z^2\EE(z_k+s(z_{k+1}-z_k))[\d_\tau z_{k+1}-\d_\tau z_k]}{\d_\tau z_{k+1}- \d_\tau z_k}_{\ZZ^\ast,\ZZ}\, \d s 
 \geq 0 .\label{eq:auxA02} 
\end{equation}
However \eqref{eq:timeprogress} and \eqref{eq:auxA02} remain valid, if $\II(t_{k},\cdot)$ is only $\kappa$-uniformly convex on a ball
$B_\ZZ(z,\Delta)$ with radius $\Delta > \tau > 0$ and $z_k,z_{k+1} \in B_\ZZ(z,\Delta)$. 
To see this, note that \eqref{eq:timeprogress} follows from estimate \eqref{eq:eq.aux002}, see proof of Lemma~\ref{lem:timeprogress}, 
which itself is a consequence of $ \dual{D_z\II(t_{k},z_{k+1})-D_z\II(t_{k},z_{k})}{z_{k+1}-z_{k}}_{\ZZ^\ast,\ZZ} \geq \kappa \norm{z_{k+1}-z_{k}}_\ZZ^2 $.
This inequality, just as inequality \eqref{eq:auxA02}, only require that $z_k$ and $z_{k+1}$ lay in a region of uniform convexity of $\II$. 
The estimate on $\gamma$ finally necessitates that $ \tilde{z}_\tau(t) \in B_\ZZ(z(t),\Delta)$ and that $\II$ is uniformly convex on $B_\ZZ(z(t),\Delta)$ 
for all $ t \in [0,T]$, cf.~the definition of $\gamma$ in \eqref{eq:defgamma}. 
\end{remark}

\begin{remark}
    In view of the regularity of the differential solution, i.e., $z\in W^{1,\infty}(0,T;\ZZ)$, the convergence rate of $\Landau{\tau}$ in 
    Theorem~\ref{thm:reverse_approx_small} can be regarded as optimal, since the piecewise affine interpolation of the solution 
    does not yield a better convergence rate.
\end{remark}

\begin{remark}
    We expect that a spatial discretization can also be included in the above a priori estimates, following e.g.~the lines of 
    \cite{mpps:errorRIS}. This would however go beyond the scope of the paper and is subject to future research.
\end{remark}

\subsection{The General Case (w/o smallness assumption on $|\ell|_{Lip}$)}\label{sec:GlobConvGen}
Let us now turn to the general case, where the Lipschitz-constant does not necessarily fulfill $|\ell|_{Lip} < \kappa $. In this case, we can no longer guarantee that the algorithm always makes progress w.r.t.~time, which implies that the back-transformation onto the physical time need not exist. In order to handle these cases, we will neglect all iterates for which the time-update does not proceed. Consequently, we need to ensure that the algorithm only needs a finite number of iterates (independent of $\tau$) to reach a new local minimum in the interior of $B_\VV(z_{k-1},\tau)$ so that, after a maximum number of $M$ iterates, the algorithm again performs a timestep. This is part of the next two Lemmata. 

\begin{lemma}\label{lem:maxNum_noSteps}
Let \cref{ass:globconv} hold. Then there exists $m \in \N$, independent of $\tau$, such that, for all iterates $k \in \N$, $k < \hat{N}(\tau)$, there exists an index $\hat{k} \in [k,k+m]$ so that $0 \in \partial\RR(0)+D_z\II(t_{\hat{k}-1},z_{\hat{k}})$,
i.e., after at most $m$ iterations, the iterate is again locally stable.
\end{lemma}
\begin{proof}
W.l.o.g. let $k$ be the last iterate with $t_k-t_{k-1} >0$ (otherwise we choose $\tilde{k}<k$ as the last iterate, where a time-step took place and apply the same argumentation with $\tilde{k}$ instead of $k$, which will then give the same $m$).
By Remark~\ref{rem:z1=z0} we have $t_1-t_0>0$ so that there always exists such an index $k \leq N(\tau)$. We will first show that $\lambda_{k+1}$ is bounded by the Lipschitz-constant of $\ell$. Afterwards, we will show that the sequence $\{\lambda_{k+l}\}_{l\geq 1}$ is monotonically decreasing by some constant factor. Since all multipliers are non-negative, this will lead to $\lambda_{k+m} = 0$, which yields the desired result. \\
\textsl{Step 1: Boundedness of $\lambda_{k+1}$} \\
Since $t_k-t_{k-1}>0$, we have $\lambda_k = 0$ by \eqref{eq:tupdate} and \eqref{eq:opt.prop01} so that \cref{lem:general_estimate} implies 
\begin{multline*} 
0 \geq \kappa \norm{z_{k+1}-z_{k}}_\ZZ^2 - |\ell|_{Lip} (t_{k}-t_{k-1}) \norm{z_{k+1}-z_{k}}_\VV + \lambda_{k+1}  \tau^2 \\
\geq - |\ell|_{Lip} (t_{k}-t_{k-1}) \norm{z_{k+1}-z_{k}}_\VV + \lambda_{k+1}  \tau^2 
\geq - |\ell|_{Lip} \tau^2 + \lambda_{k+1} \tau^2
\end{multline*}
so that indeed $\lambda_{k+1} \leq |\ell|_{Lip}$.\\
\textsl{Step 2: Monotonicity of $\{\lambda_{k+l}\}_{l \geq 1}$}\\
To proceed, let $l\geq 2$ iterations be given without time-progress (otherwise $m=2$), which means that 
\begin{gather}
t_{k+l} = t_{k+l-1} = \dots = t_k  \label{eq:eq.aux003} \\
\text{and} \quad \norm{z_{k+l}-z_{k+l-1}}_\VV = \norm{z_{k+l-1}-z_{k+l-2}}_\VV = \dots = \tau . \label{eq:eq.aux006}
\end{gather}
We will now show that the sequence $\{\lambda_{k+l}\}_{l\geq 1}$ is monotonically decreasing by some constant factor. Together with \eqref{eq:eq.aux002} for the index $k+l$, \eqref{eq:eq.aux003} implies
\[ 0 \geq \kappa \norm{z_{k+l}-z_{k+l-1}}_\ZZ^2 + \lambda_{k+l} \tau^2 - \lambda_{k+l-1} \tau^2  . \]
Using the embedding $\ZZ \embeds \VV$ and inserting \eqref{eq:eq.aux006}, we obtain
$0 \geq \kappa \tau^2 + \lambda_{k+l} \tau^2 - \lambda_{k+l-1} \tau^2$.
Combining this with the bound on $\lambda_{k+1}$ from above and rearranging terms then yields
\[ \lambda_{k+l} \leq \lambda_{k+l-1} - \kappa \quad \Longrightarrow \quad \lambda_{k+l} \leq \lambda_{k+1} - (l-1) \kappa \leq |\ell|_{Lip} - (l-1) \kappa ,\]
which finally gives that $\lambda_{k+m} = 0$ for $m = \lceil |\ell|_{Lip}/\kappa \rceil +1$ due to the non-negativity of the multipliers. Thus, by \eqref{eq:discrOptCon_DiffInc}, we have 
$0 \in \partial\RR(0) + D_z\II(t_{k+m-1},z_{k+m})$.
\end{proof}

\begin{lemma}\label{lem:maxNum_noSteps2}
Let \cref{ass:globconv} hold. Then there exists $M \in \N$, independent of $\tau$ and $\varepsilon$, such that, for all iterates $k \in \N$, $k < N(\tau)$, there exists an index $\hat{k} \in [k,k+M]$ so that $t_{\hat{k}+1} - t_{\hat{k}} > 0$, 
i.e., after at most $M$ iterations, the algorithm performs a timestep.
\end{lemma}
\begin{proof}
From \cref{lem:maxNum_noSteps} there exists $m \in \N$ such that 
\begin{equation}\label{eq:auxA07} 
0 \in \partial\RR(0) + D_z\II(t_{k+m-1},z_{k+m}) . 
\end{equation} 
Therefore, it either holds $\norm{z_{k+m}-z_{k+m-1}}_\VV < \tau$, which implies that $t_{k+m}-t_{k+m-1} > 0$, or $\norm{z_{k+m}-z_{k+m-1}}_\VV = \tau$ and \eqref{eq:auxA07} in combination with the time-update \eqref{eq:tupdate} implies that 
\begin{equation*} 
\norm{z_{k+m}-z_{k+m-1}}_\ZZ \leq \frac{|\ell|_{Lip}}{\kappa} (t_{k+m} - t_{k+m-1}) 
= \frac{|\ell|_{Lip}}{\kappa} ( \tau - \norm{z_{k+m}-z_{k+m-1}}_\VV ) = 0 .
\end{equation*}
Again, from the time-update \eqref{eq:tupdate}, it follows $t_{k+m+1}-t_{k+m}=\tau>0$. In both cases, we have proven the assertion for $M = m+1$.
\end{proof}

%

\noindent
We finally need an estimate for the iterates in the stronger $\ZZ$-norm, in order to get a uniform bound for the derivative of the linear-interpolants.

\begin{lemma}\label{lem:bound_iter_Z}
Let \cref{ass:globconv} be satisfied. Then there exists a constant $C = C(|\ell|_{Lip},\kappa)>0$ such that
$\norm{z_{k} - z_{k-1}}_\ZZ \leq C \, \tau$ for all iterations $k \leq \hat{N}(\tau)$. 
\end{lemma}
\begin{proof} For $k=1$ this easily follows from Remark~\ref{rem:z1=z0}. Hence, let $k \geq 2$. In the proof of \cref{lem:maxNum_noSteps}, we have seen that the multipliers $\lambda_k$ are bounded by $|\ell|_{Lip}$ for all $k \leq \hat{N}(\tau)$. Another application of \cref{lem:general_estimate} thus gives
\begin{multline*}
\kappa \norm{z_{k}-z_{k-1}}_\ZZ^2 
\leq |\ell|_{Lip} (t_{k-1}-t_{k-2}) \norm{z_{k}-z_{k-1}}_\VV - (\lambda_{k} - \lambda_{k-1}) \tau^2 \\
\leq |\ell|_{Lip} \tau^2 + \lambda_{k-1}  \tau^2  
\leq 2 |\ell|_{Lip} \tau^2 ,
\end{multline*} 
where we exploited the positivity of the multiplier $\lambda_{k}$.
\end{proof}

As mentioned above, the time-discrete parametrized solution will only include the iterates for which the time-update proceeds. Thus we set
\begin{itemize}
\item $N(\tau) =$ number of iterations to reach the end-time $T$ (with stepsize $\tau$)
\item $\hat{N}(\tau) =$ number of iterations to reach the final locally stable state $z_{\hat{N}(\tau)}$ (see Remark~\ref{rem:hatN})
\item $\NN(\tau) := \{ k \in \{1,\dots,N(\tau)\} \, : \, t_{k}-t_{k-1} > 0 \} \cup \{0,\hat{N}(\tau)\}$
\end{itemize}

In what follows, the iterations in $\NN(\tau)$ are numbered from $1$ to $|\NN(\tau)|$ and the corresponding mapping is denoted by $\kindex$, i.e.,
\begin{align*}
\kindex : \{0,1,\dots,|\NN(\tau)| \} \to \NN(\tau) \quad \text{so that} \quad \NN(\tau) = \{\kindex(0), \kindex(1) , \dots, \kindex(\NN(\tau)) \}.
\end{align*}
Therewith, we define for $t \in [t_{\kindex(j-1)},t_{\kindex(j)})$ 
\begin{equation*} 
\tilde{z}_\tau(t) = z_{\kindex(j)} + \frac{t-t_{\kindex(j-1)}}{t_{\kindex(j)}-t_{\kindex(j-1)}} \left( z_{\kindex(j+1)}-z_{\kindex(j)} \right) , \quad \tilde{z}_\tau(T) = z_{\hat{N}(\tau)}
\end{equation*}
as well as $\zcu{}(t) = z_{\kindex(j)}$, $\tcl{}(t) = t_{\kindex(j)-1}$. 
\noindent Note that it holds
\begin{equation}\label{eq:auxA09}
t_{k} = \dots = t_{\kindex(j-1)}  \quad \forall k \in \{\kindex(j-1),\kindex(j-1)+1,\dots,\kindex(j)-1\}
\end{equation}
and consequently
\begin{equation}
\quad 0 \in \partial\RR(0) + D_z\II(t_{\kindex(j)-1},z_{\kindex(j)}) = \partial\RR(0) + D_z\II(\tcl{}(t),\zcu{}(t)). \label{eq:auxA06}
\end{equation}
Moreover we have the following estimates:

\begin{lemma}\label{lem:bound_z'_nosmall}
Let \cref{ass:globconv} and $-D_z\II(0,z_0) \in \partial\RR(0)$ hold. Then there exists constants $M \in \N$ and $C_1,C_2>0$ independent of $\tau$ and $\varepsilon$ so that
\begin{alignat}{3}
\kindex(j) - \kindex(j-1) &\leq M & \quad &\forall j = 1,\dots,|\NN(\tau)| \label{eq:estimate_numSmallSteps}\\
\norm{(\tilde{z}_\tau)^\prime(t)}_\ZZ &\leq C_1 & \quad &\forall_{a.a.}\, t \in [0,T], \label{eq:bound_discrete-derivative} \\
\norm{\tilde{z}_\tau(t)-\zcu{}(t)}_\ZZ &\leq C_2 \tau & \quad &\forall \, t \in [0,T], \label{eq:bound_difference-interpolants} \\
\abs{t-\tcl{}(t)} &\leq \tau & \quad &\forall \, t \in [0,T]. \label{eq:bound_backtrafo-time}
\end{alignat}
\end{lemma}
\begin{proof} The first statement is a direct consequence of \cref{lem:maxNum_noSteps}. Let $\varepsilon := \frac{\kappa}{\kappa+|\ell|_{Lip}} \leq 1$. In order to estimate the derivative of the affine interpolants, let $j \in \{1,\dots,|\NN(\tau)|-1 \}$. We then distinguish the following two cases:
\begin{enumerate}[label=\roman*)]
\item If $(t_{\kindex(j)}-t_{\kindex(j-1)}) \geq \varepsilon \tau$ then
\begin{equation}\label{eq:tildezprime1}
\vnorm{\frac{ z_{\kindex(j+1)}-z_{\kindex(j)}}{t_{\kindex(j)}-t_{\kindex(j-1)}} }_\ZZ 
\leq \sum_{i=\kindex(j-1)}^{\kindex(j)-1}\frac{ \norm{z_{i+1}-z_{i}}_\ZZ}{\varepsilon \tau}
\leq \frac{M C}{\varepsilon} .
\end{equation}
\item Otherwise $\varepsilon \tau > (t_{\kindex(j)}-t_{\kindex(j-1)}) > 0$. Since $\kindex(j) \in \mathcal{N}(\tau)$, the complementarity condition \eqref{eq:opt.prop01} and the time-update \eqref{eq:tupdate} imply $\lambda_{\kindex(j)} = 0$. Consequently, \cref{lem:estimate_consec_step} in combination with \eqref{eq:auxA09} give
\begin{equation} \norm{z_{\kindex(j)+1}-z_{\kindex(j)}}_\ZZ \leq \frac{|\ell|_{Lip}}{\kappa} (t_{\kindex(j)}-t_{\kindex(j)-1}) = \frac{|\ell|_{Lip}}{\kappa} (t_{\kindex(j)}-t_{\kindex(j-1)}) . \label{eq:auxA08} \end{equation}
Therefore, if $t_{\kindex(j)} < T$, then the time update \eqref{eq:tupdate} and the embedding $\ZZ \embeds \VV$ give 
\[ t_{\kindex(j)+1}-t_{\kindex(j)} = \tau - \norm{z_{\kindex(j)+1}-z_{\kindex(j)}}_\VV \geq (1-\frac{|\ell|_{Lip}}{\kappa} \varepsilon) \tau = \varepsilon \tau >0 \]
and consequently, $\kindex(j+1) = \kindex(j)+1$. If $ t_{\kindex(j)} = t_{N(\tau)} = T$, then \eqref{eq:auxA08} implies 
\[ \norm{z_{\kindex(j)+1}-z_{\kindex(j)}}_\VV \leq \frac{|\ell|_{Lip}}{\kappa} \varepsilon \tau < \tau \]
so that $z_{\kindex(j)+1}$ is locally stable, which in turn yields $\hat{N}(\tau) = \kindex(j)+1$ and hence $\kindex(j+1) = \hat{N}(\tau) = \kindex(j)+1$. Thus, in both cases, $\kindex(j+1) = \kindex(j)+1$ and consequently, \eqref{eq:auxA08} yields
\begin{equation}\label{eq:tildezprime2}
\vnorm{\frac{ z_{\kindex(j+1)}-z_{\kindex(j)}}{t_{\kindex(j)}-t_{\kindex(j-1)}} }_\ZZ \leq \frac{|\ell|_{Lip}}{\kappa} . 
\end{equation}
\end{enumerate}
Hence, \eqref{eq:tildezprime1} and \eqref{eq:tildezprime2} give \eqref{eq:bound_discrete-derivative} with $C_1 = \max\{\frac{M C (\kappa + |\ell|_{Lip})}{\kappa},\frac{|\ell|_{Lip}}{\kappa}\}$. For \eqref{eq:bound_difference-interpolants}, we first calculate
\[ \norm{\tilde{z}_\tau(t)-\zcu{}(t)}_\ZZ = \vabs{\frac{t-t_{\kindex(j-1)}}{t_{\kindex(j)}-t_{\kindex(j-1)}}}\norm{z_{\kindex(j+1)}-z_{\kindex(j)}}_\ZZ . \] 
Another application of \eqref{eq:estimate_numSmallSteps} and \cref{lem:bound_iter_Z} then yield for all $t \in [0,T]$
\[ \norm{\tilde{z}_\tau(t)-\zcu{}(t)}_\ZZ \leq \sum_{i=\kindex(j)}^{\kindex(j+1)-1} \norm{z_{i+1}-z_i}_\ZZ\leq M C \tau =: C_2 \tau. \]
Finally \eqref{eq:bound_backtrafo-time} is a direct consequence of the construction of $\tcl{}(t)$.
\end{proof}

\begin{remark}
Taking a closer look to the proof of \cref{lem:bound_z'_nosmall} we observe that it actually holds
\begin{equation} \frac{1}{t_{\kindex(j)}-t_{\kindex(j-1)}}\sum_{i=\kindex(j)}^{\kindex(j+1)-1} \norm{z_{i+1}-z_i}_\ZZ \leq C  \label{eq:iota_sum_bound} 
\end{equation}
for all $ j < \hat{N}(\tau)$.
\end{remark}

With all this at hand, we are now in the position to show an a-priori estimate in the general case:
\begin{theorem}\label{thm:reverse_approx_general}
Let \cref{ass:globconv} be fulfilled. Then there exists $C>0$, independent of $\tau$, such that for the affine interpolants $\tilde{z}_\tau:[0,T]\to \ZZ$, defined as above, it holds:
\[ \norm{z(t)-\tilde{z}_\tau(t)}_\ZZ \leq C \sqrt{\tau} \quad \forall t \in [0,T], \]
where $z \in C^{0,1}([0,T];\ZZ)$ is the unique (differential) solution of the RIS.
\end{theorem}
\begin{proof}
First of all, from \cref{thm:app.existence_and_uniqueness} we have the existence of a unique differential solution $z \in C^{0,1}(0,T;\ZZ)$, that fulfills for all $v \in \ZZ$
\begin{equation}
\RR(z^\prime(t)) \geq \RR(v) + \dual{-D_z\II(t,z(t))}{v-z^\prime(t)}_{\ZZ^\ast,\ZZ} \qquad \faa t \in [0,T].
\label{eq:thm.uniqueVISol}
\end{equation}
On the other hand, according to \eqref{eq:auxA06}, we have for all $v\in \ZZ$ that
\begin{equation}
-D_z\II(\tcl{}(t),\zcu{}(t)) \in \partial\RR(0) \quad \Longleftrightarrow \quad \RR(v) \geq \dual{-D_z\II(\tcl{}(t),\zcu{}(t)}{v}_{\ZZ^\ast,\ZZ}. \label{eq:back-transformedInc1}
\end{equation}
Moreover, for $t \in [t_{\kindex(j-1)},t_{\kindex(j)} )$, the positive homogeneity and convexity of $\RR$ together with \eqref{eq:opt.prop03} give
\begin{align*}
 \RR(\tilde{z}^\prime_\tau(t)) &= \RR \left( \frac{ z_{\kindex(j)}-z_{\kindex(j-1)}}{t_{\kindex(j)}-t_{\kindex(j-1)}} \right) \leq \frac{1}{t_{\kindex(j)}-t_{\kindex(j-1)}} \sum_{i = \kindex(j-1)}^{\kindex(j)-1}\RR(z_{i+1}-z_{i}) \\ 
 &\leq \frac{1}{t_{\kindex(j)}-t_{\kindex(j-1)}} \sum_{i = \kindex(j-1)}^{\kindex(j)-1} \dual{-D_z\II(t_i,z_{i+1})}{z_{i+1}-z_{i}}_{\ZZ^\ast,\ZZ} \\
&= \dual{-D_z\II(\tcl{}(t)),\zcu{}(t)}{\tilde{z}_\tau^\prime(t)}_{\ZZ^\ast,\ZZ} \\
&\quad + \frac{1}{t_{\kindex(j)}-t_{\kindex(j-1)}} 
\sum_{i = \kindex(j-1)}^{\kindex(j)-1} \dual{D_z\II(\tcl{}(t),\zcu{}(t))-D_z\II(t_i,z_{i+1})}{z_{i+1}-z_{i}}_{\ZZ^\ast,\ZZ} .
\end{align*}
For the last term, we further estimate
\begin{align*}
&\dual{D_z\II(\tcl{}(t),\zcu{}(t))-D_z\II(t_i,z_{i+1})}{z_{i+1}-z_{i}}_{\ZZ^\ast,\ZZ}  \\
&\quad \leq \dual{A(\zcu{}(t)-z_{i+1})}{z_{i+1}-z_i}_{\ZZ^\ast,\ZZ} + \dual{D_z\FF(\zcu{}(t))-D_z\FF(z_{i+1})}{z_{i+1}-z_i}_{\ZZ^\ast,\ZZ} \\
&\quad \leq \norm{A}_{\LL(\ZZ,\ZZ^\ast)} \norm{z_{\kindex(j)}-z_{i+1}}_\ZZ \norm{z_{i+1}-z_i}_\ZZ  + C_\FF \norm{z_{\kindex(j)}-z_{i+1}}_\ZZ \norm{z_{i+1}-z_i}_\ZZ \\
&\quad \leq C \, \tau \norm{z_{i+1}-z_i}_\ZZ ,
\end{align*}
where we used \cref{lem:ehrling}, \cref{lem:bound_iter_Z}, \eqref{eq:estimate_numSmallSteps}, and the fact that $t_i = t_{\kindex(j)} = \tcl{}(t)$ for all $ i \in \{\kindex(j-1),\dots,\kindex(j)-1\}$, see \eqref{eq:auxA09}. Exploiting \eqref{eq:iota_sum_bound}, and combining the resulting estimate with \eqref{eq:back-transformedInc1} gives for all $w \in \ZZ$:
\begin{equation}
\RR(w) - \RR(\tilde{z}^\prime_\tau(t)) + \dual{D_z\II(\tcl{}(t),\zcu{}(t))}{w - \tilde{z}^\prime_\tau(t)}_{\ZZ^\ast,\ZZ} \geq -C \, \tau \quad \faa t \in [0,T] \, . \label{eq:auxA05}
\end{equation}
Testing \eqref{eq:thm.uniqueVISol} with $v=\tilde{z}^\prime_\tau(t)$ and \eqref{eq:auxA05} with $w =z^\prime(t)$, respectively, and summing up the resulting inequalities yields
\begin{align*}
&C \, \tau \geq \dual{D_z\II(\tcl{}(t),\zcu{}(t)) - D_z\II(t,z(t))}{\tilde{z}^\prime_\tau(t)-z^\prime(t)}_{\ZZ^\ast,\ZZ} \\
&\; = \langle D_z\II(\tcl{}(t),\zcu{}(t)) - D_z\II(t,\zcu{}(t)) \\
&\qquad + D_z\II(t,\zcu{}(t)) - D_z\II(t,\tilde{z}_\tau(t)) 
 + D_z\II(t,\tilde{z}_\tau(t)) - D_z\II(t,z(t)),\tilde{z}^\prime_\tau(t)-z^\prime(t)\rangle_{\ZZ^\ast,\ZZ} .
\end{align*}
Since $z$ is Lipschitz continuous, we have $\norm{z^\prime(t)}_\ZZ \leq C$ a.e. in $[0,T]$. In combination with \eqref{eq:bound_discrete-derivative} as well as \cref{lem:ehrling} (note that $\tilde{z}_\tau$ and $ \zcu{}$ are bounded independent of $\tau$), we can thus estimate
\begin{align}
&\hspace*{-1cm}\dual{D_z\II(t,\tilde{z}_\tau(t)) - D_z\II(t,z(t))}{\tilde{z}^\prime_\tau(t)-z^\prime(t)}_{\ZZ^\ast,\ZZ} \notag \\
&\leq \vabs{\dual{D_z\II(\tcl{}(t),\zcu{}(t))- D_z\II(t,\zcu{}(t))}{\tilde{z}^\prime_\tau(t)) - z^\prime(t)}_{\ZZ^\ast,\ZZ}} \notag \\
&\quad + \vabs{\dual{D_z\II(t,\zcu{}(t)) - D_z\II(t,\tilde{z}_\tau(t))}{\tilde{z}^\prime_\tau(t)) - z^\prime(t)}_{\ZZ^\ast,\ZZ}} + C \, \tau \notag \\
&\leq \norm{\ell(\tcl{}(t))-\ell(t)}_\VV \norm{\tilde{z}_\tau^\prime(t) - z^\prime(t)}_\VV \notag \\
&\quad + \vabs{\dual{A\zcu{}(t) - A\tilde{z}_\tau(t)}{\tilde{z}^\prime_\tau(t) - z^\prime(t)}_{\ZZ^\ast,\ZZ}} \notag \\
&\quad + \vabs{\dual{D_z\FF(\zcu{}(t)) - D_z\FF(\tilde{z}_\tau(t))}{\tilde{z}^\prime_\tau(t)) - z^\prime(t)}_{\ZZ^\ast,\ZZ}} + C \, \tau \notag \\
&\begin{aligned}[t] \leq \Big( C \norm{\zcu{}(t) - \tilde{z}_\tau(t))}_\ZZ + C_{\FF} 
\norm{\zcu{}(t) - \tilde{z}_\tau(t)}_\ZZ &\notag \\[-1ex]
+ |\ell|_{Lip} \abs{\tcl{}(t)-t} & \Big) \, \norm{\tilde{z}_\tau^\prime(t) - z^\prime(t)}_\ZZ + C \, \tau \end{aligned} \notag \\
&\leq C \tau \, (\norm{\tilde{z}_\tau^\prime(t)}_\ZZ + \norm{z^\prime(t)}_{\ZZ}) + C \, \tau 
\leq C \, \tau, \label{eq:auxA10}
\end{align}
%
where we used \eqref{eq:bound_difference-interpolants} and \eqref{eq:bound_backtrafo-time} in the next-to-last inequality. We can now in principle follow the lines of \citep[Thm 7.4]{mielketheil}. Since an additional error $ C \tau $ arise in \eqref{eq:auxA10}, we need to adapt some estimates of \cite{mielketheil} and therefore we give the main details:
Again we define the error measure
$\gamma(t) := \dual{D_z\II(t,\tilde{z}_\tau(t)) - D_z\II(t,z(t))}{\tilde{z}_\tau(t)-z(t)}_{\ZZ^\ast,\ZZ}$.
Due to the $\kappa$-uniform convexity of $\II(t,\cdot)$, we have $\gamma(t) \geq \kappa \norm{\tilde{z}_\tau(t)-z(t)}_\ZZ^2$. In full analogy to \citep[Thm 7.4]{mielketheil}, we can estimate (see Appendix~\ref{sec:App.errorGamma})
\begin{equation*}
\dot{\gamma}(t) \leq C \norm{\tilde{z}_\tau(t)-z(t)}_\ZZ^2
+ 2 \dual{D_z\II(t,\tilde{z}_\tau(t)) - D_z\II(t,z(t))}{\tilde{z}^\prime_\tau(t)-z^\prime(t)}_{\ZZ^\ast,\ZZ} \, ,
\end{equation*}
wherein we use the essential boundedness of $\tilde{z}_\tau^\prime$ and $z^\prime$. Inserting \eqref{eq:auxA10} and exploiting that $\gamma(t)/\kappa \geq \norm{\tilde{z}_\tau(t)-z(t)}_\ZZ^2$ we obtain 
$\dot{\gamma}(t) \leq C \gamma(t) + C \tau$.
Now, we proceed as in the end of the proof of \cref{thm:reverse_approx_small}. Integrating and using Gronwall's Lemma yields
$\gamma(t) \leq (\gamma(0) + C T \tau)\exp^{C t} \leq C (\gamma(0)+\tau)$.
Due to $\hat{z}(0) = z(0) = z_0$, we have $\gamma(0) = 0$. Exploiting again the $\kappa$-uniform convexity of $\II$, we finally obtain $\norm{\tilde{z}_\tau(t)-z(t)}_\ZZ^2 \leq \gamma(t)/\kappa \leq C \, \tau$,
which was claimed.
\end{proof}

\begin{remark}\label{rem:notoptimal}
    In contrast to Theorem~\ref{thm:reverse_approx_small}, we do not obtain the optimal rate of convergence  
    in case the Lipschitz constant of $\ell$ is too big. 
    The critical part of the proof is the estimate of 
    $\sum_{i = \kindex(j-1)}^{\kindex(j)-1} \dual{D_z\II(\tcl{}(t),\zcu{}(t))-D_z\II(t_i,z_{i+1})}{z_{i+1}-z_{i}}$, 
    that only yields an order $\Landau{\tau}$ instead of $\Landau{\tau^2}$, which would be necessary to obtain the optimal order. 
    A potential resort could be to replace $\tilde z_\tau$ by a more sophisticated interpolant that does not simply neglect all iterations without progress 
    in the physical time.
    Note that, due to the $1-$homogeneity of the dissipation, it is always possible to achieve $|\ell|_{Lip} < \kappa$ by rescaling the time accordingly. 
    Then, Theorem~\ref{thm:reverse_approx_local} applies giving the optimal order in the rescaled time scale. Of course, depending on the 
    Lipschitz constant of $\ell$, the rescaled time scale might become rather small so that a large number of iterations is necessary, but 
    this rescaling argument indicates that it should be possible to achieve the optimal order in case of large $|\ell|_{Lip}$, too. 
    This however gives rise to future research.
\end{remark}


\subsection{A priori Analysis for Locally Uniformly Convex Energies}\label{sec:LocConv}

As already mentioned in the introduction, the local incremental minimization algorithm is actually not necessary, if the energy is globally uniformly convex. 
In this case, one could also use the global incremental minimization scheme, which is easier to implement, since the additional inequality constraint 
in \eqref{eq:locmin} is omitted. The situation changes however, if the energy is no longer globally uniformly convex, but only locally around a 
given evolution $z$. Then the local incremental minimization scheme still approximates the (local) solution with optimal order (provided that $|\ell|_{Lip}$
is not too large), while the global scheme might fail to converge, as we will demonstrate by means of a numerical example in Section~\ref{sec:testloc}.
Our precise notion of local uniform convexity is as follows:

\begin{assumption}[Local $\kappa$-uniform convexity]\label{ass:locconv}\\
We call $\II$ \emph{locally $\kappa$-uniform convex} around $ z :[0,T] \to \ZZ$ if there exist $\kappa,\Delta>0$, independent of $t$, such that $\II(t,\cdot)$ is $\kappa$-uniformly convex on $\overbar{B_\ZZ(z(t), \Delta)}$ for all $t \in [0,T]$, i.e.
\begin{equation}\label{eq:ineq_locconv} 
\dual{D_z^2\II(t,\tilde{z})v}{v}_{\ZZ^\ast,\ZZ} \geq \kappa \norm{v}_\ZZ^2 \qquad \forall \tilde{z} \in \overbar{B_\ZZ(z(t),\Delta)}, \, v \in \ZZ . 
\end{equation}
\end{assumption}
\noindent Note that local uniform convexity is always referred to an evolution $z$. The \cref{ass:locconv} especially implies that
\begin{equation}\label{eq:ineq_locconv2}
\dual{D_z\II(t,z_2)-D_z\II(t,z_1)}{z_2-z_1}_{\ZZ^\ast,\ZZ} \geq \kappa \norm{z_2-z_1}_\ZZ^2 \qquad \forall z_1,z_2 \in \overbar{B_\ZZ(z(t),\Delta)}
\end{equation}
Indeed, using \eqref{eq:ineq_locconv}, we obtain
\begin{align*}
&\dual{D_z\II(t,z_2)-D_z\II(t,z_1)}{z_2-z_1}_{\ZZ^\ast,\ZZ} \\
&\qquad = \int_0^1 \dual{D_z^2\II(t,z_1+s(z_2-z_1))[z_2-z_1]}{z_2-z_1}_{\ZZ^\ast,\ZZ} \d s
\geq \kappa \norm{z_2-z_1}_\ZZ^2
\end{align*}
where we used that $z_1+s(z_2-z_1) \in \overbar{B_\ZZ(z(t),\Delta)}$ for all $ s \in [0,1]$. Now, in order to prove a convergence-rate in the local uniform convex case, we again have to estimate the difference of iterates in the $\ZZ$-norm. Since it is not a-priori clear that the iterate remains in the neighbourhood of convexity of $\II$, we need to alter the proof of \cref{lem:bound_iter_Z}.

\begin{lemma}\label{lem:bound_discrete-loc-Z}
Let $0 \in \partial\RR(0) + D_z\II(t_{k-1},z_k)$ for some $k \in \N$. Then 
$\norm{z_{k+1} - z_{k}}_\ZZ \leq C_{loc} \, \tau$.
for some constant $C_{loc} = C_{loc}(\FF,\alpha,|\ell|_{Lip})>0$.
\end{lemma}
\begin{proof} Let $k\in \N$ be given. From \eqref{eq:eq.general_estimate_2} we know
\begin{align*}
0 &\geq \dual{D_z\II(t_{k},z_{k+1})-D_z\II(t_{k},z_{k})}{z_{k+1}-z_{k}}_{\ZZ^\ast,\ZZ} \notag \\
&\quad + \dual{D_z\II(t_{k},z_{k})-D_z\II(t_{k-1},z_{k})}{z_{k+1}-z_{k}}_{\ZZ^\ast,\ZZ}  + (\lambda_{k+1}-\lambda_k) \tau^2 \end{align*}
Since $0 \in \partial\RR(0) + D_z\II(t_{k-1},z_k)$ holds by assumption, \eqref{eq:lambda_dist_expression} implies $\lambda_k = 0$. Inserting the definition of $\II$ and exploiting Remark~\ref{rem:ehrling2}, we can thus further estimate
\begin{align*}
0 &\geq \dual{A(z_{k+1}-z_k)}{z_{k+1}-z_k}_{\ZZ^\ast,\ZZ} + \dual{D_z\FF(z_{k+1})-D_z\FF(z_k)}{z_{k+1}-z_k}_{\ZZ^\ast,\ZZ} \\ 
&\quad + \dual{\ell(t_{k-1})-\ell(t_k)}{z_{k+1}-z_k} + \lambda_{k+1} \tau^2 \\
&\geq \alpha \norm{z_{k+1}-z_k}_\ZZ^2 - C_\FF \norm{z_{k+1}-z_k}_\ZZ \norm{z_{k+1}-z_k}_\VV 
 - |\ell|_{Lip} (t_k-t_{k-1}) \norm{z_{k+1}-z_k}_\VV . 
\end{align*}
Therefor, by applying the generalized Young-inequality, it follows from the constraint in \eqref{eq:locmin} that
\begin{align*}
0 &\geq \alpha \norm{z_{k+1}-z_k}_\ZZ^2 - \frac{\alpha}{2} \norm{z_{k+1}-z_k}_\ZZ^2 -C_{\FF,\alpha} \norm{z_{k+1}-z_k}_\VV^2 - |\ell|_{Lip}\, \tau^2 \\
&\geq \frac{\alpha}{2} \norm{z_{k+1}-z_k}_\ZZ^2 - C_{\FF,\alpha} \tau^2 - |\ell|_{Lip}\, \tau^2 
\end{align*}
so that indeed $C_{loc} \, \tau^2 \geq  \norm{z_{k+1}-z_k}_\ZZ^2$
with $C_{loc} = \frac{2}{\alpha}(C_{\FF,\alpha} + |\ell|_{Lip} )$.
\end{proof}
With this at hand, we can now show an a-priori estimate in the case of an energy-functional, which is only locally uniform convex around a differential solution.

\begin{theorem}\label{thm:reverse_approx_local}
Let $z \in C^{0,1}([0,T];\ZZ)$ be a (differential) solution. 
Furthermore let $\II$ be locally $\kappa$-uniform convex around $z$ with radius $\Delta>0$ and assume that $\ell \in W^{1,\infty}([0,T];\VV)$ with $\abs{\ell}_{Lip} \leq \kappa - \delta$ (see \cref{ass:Lipschitz_small}) and $\ell^\prime \in BV([0,T];\VV)$. Then there exists a constant $K_{loc}>0$, independent of $\tau$, such that, for the back-transformed parameterized solution $z_\tau:[0,T]\to \ZZ$ and all $\tau \leq \bar{\tau}$ with $\bar{\tau}$ sufficiently small, it holds:
\begin{equation} 
\norm{z_\tau(t)-z(t)}_\ZZ \leq K_{loc} \, \tau \quad \forall t \in [0,T]. \label{eq:eq.reverse_approx_local}
\end{equation}
\end{theorem}

\begin{proof} 
The proof basically follows the Steps in the proof of \cref{thm:reverse_approx_small}. 
Though we need to ensure that the iterates remain in the region of uniform convexity of $\II$, see Remark~\ref{rem:uni-conv_estimates}. 
Therefor, we will show by means of induction, that $z_k,z_{k+1} \in B_\ZZ(z(t),\Delta)$ for $t \in [t_{k-1},t_k]$. 
As an easy consequence, the affine interpolant $\tilde{z}_\tau$, defined in \eqref{eq:affine_thm_loc} below, 
fulfills $\tilde{z}_\tau(t) \in B_\ZZ(z(t),\Delta)$ for $t \in [t_{k-1},t_k]$, 
which yields that the estimates in Remark~\ref{rem:uni-conv_estimates} also hold in the local convex case and we can proceed as in the proof of \cref{thm:reverse_approx_small}.\\[0.5ex]
\textsl{Step 0: Preparation}\\ We start by choosing 
\begin{equation}
\tau \leq \min\left(\frac{\Delta}{3 \, C_{loc}} , \frac{\Delta}{3 \, K^\prime} , \frac{\Delta}{3 \, |z|_{Lip}}, \frac{\Delta}{3}\right) =: \bar{\tau}, \label{eq:tau_choice} 
\end{equation} where $C_{loc}$ denotes the constant from \cref{lem:bound_discrete-loc-Z} and $K^\prime$ the constant from \cref{thm:reverse_approx_small}. To be precise here, assume that $\II$ is globally $\kappa$-uniform convex. Then, by \cref{thm:reverse_approx_small}, there would exist a constant $K^\prime$ such that the a-priori estimate \eqref{eq:reverse_approx_small} would hold on $[0,T]$. This is the constant we refer to here. To proof \eqref{eq:eq.reverse_approx_local}, we will now successively show that the affine-interpolant defined by 
\begin{equation}\label{eq:affine_thm_loc}
\tilde{z}_\tau(t) := z_{k} + \frac{t-t_{k-1}}{t_k-t_{k-1}} ( z_{k+1} - z_{k} ) \quad t \in [t_{k-1},t_k), 
\end{equation}
fulfills \eqref{eq:eq.reverse_approx_local} on every interval $[t_{k-1},t_{k}]$. Since we might have $[t_{k-1},t_k) = \emptyset$, 
this definition is at first only formally. However, we will successively show by means of induction w.r.t $k$, that $t_k - t_{k-1} \geq \varepsilon \tau$ for some $\varepsilon \in [0,1)$ independent of $\tau$.\\[0.5ex]
\textsl{Step 1: Initialization}\\ 
We show \eqref{eq:eq.reverse_approx_local} for $t \in [t_0,t_1]$. To do so, we observe that, due to the choice of $\tau$, we have $B_\ZZ(z_0,\tau) \subset B_\ZZ(z_0,\Delta)$. Hence, $\II(0,\cdot)$ is convex on $B_\ZZ(z_0,\tau)$ and consequently, we can argue exactly as in Remark~\ref{rem:z1=z0} 
to obtain $z_1 = z_0 \in B_\ZZ(z(0),\Delta)$ and $t_1 - t_0 = \tau$ so that $\tilde{z}_\tau$ is well defined and equals $z_0$ on $[t_0,t_1]$. 
Since $z_0, z_1 \in B_\ZZ(z(0),\Delta)$ and $\II(t_0, \cdot)$ is uniformly convex there by assumption, 
the estimates \eqref{eq:timeprogress} and \eqref{eq:auxA02} hold for $k=1$ (see Remark~\ref{rem:uni-conv_estimates}).
Moreover, we obviously have $\tilde{z}_\tau(t) \equiv z_0 \in B_\ZZ(z(t),\Delta)$ for all $t \in [t_0,t_1]$, 
due to the Lipschitz-continuity of $z$ and the choice of $\tau$. 
Therefore, we can exploit the convexity of $\II(t,\cdot)$ on $B_\ZZ(z(t),\Delta)$, giving that $\eqref{eq:gammaest}$ holds 
for $t\in [t_0, t_1]$, too. Then, as illustrated in Remark~\ref{rem:uni-conv_estimates}, 
we can argue analogous to the proof of \cref{thm:reverse_approx_small} (steps 2--6) to obtain
$\|\tilde z_\tau(t) - z(t)\|_\ZZ \leq K' \,\tau$ for all $t\in [t_0, t_1]$. \\[0.5ex]
\textsl{Step 2: Induction}\\ Let $k \in \N$ be given with 
\begin{gather} 
z_k \in B_\ZZ(z(t_{k-1}),\Delta), \quad \norm{z_k-z_{k-1}}_\VV < \tau, \label{eq:induction1}\\
\norm{\tilde{z}_\tau(t)-z(t)}_\ZZ \leq K^\prime \, \tau \quad\forall\,t \in [t_0,t_k]. \label{eq:eq.aux009}
\end{gather} 
In the first step of the proof, we have seen that these conditions are fulfilled for $k=1$ and 
will now show that we can then extend these estimates to $[t_0,t_{k+1}]$. 
For this, we observe that, since $\tau \leq \frac{\Delta}{3 \, K^\prime}$, 
the inequality \eqref{eq:eq.aux009} gives $z_k = \tilde{z}_\tau(t_k) \in B_{\Delta/3}(z(t_k))$. 
Thus, by exploiting \cref{lem:bound_discrete-loc-Z} and \eqref{eq:tau_choice}, it follows that 
$\norm{z_{k+1}-z(t_k)}_\ZZ \leq \Delta$ so that again 
the estimates \eqref{eq:timeprogress} and \eqref{eq:auxA02} hold true (see Remark~\ref{rem:uni-conv_estimates}). 

It remains to show that $\tilde{z}_\tau(t) \in B_\ZZ(z(t),\Delta)$ for all $t \in [t_k,t_{k+1}]$ so that we have 
\eqref{eq:gammaest} on the next time interval, see again Remark~\ref{rem:uni-conv_estimates}.
By \eqref{eq:induction1}, it holds $\lambda_k=0$ such 
that the inequality \eqref{eq:eq.general_estimate_2}, in combination with $\lambda_{k+1} \geq 0$, reduces to 
\begin{multline*}
0 \geq \dual{D_z\II(t_{k},z_{k+1})-D_z\II(t_{k},z_{k})}{z_{k+1}-z_{k}}_{\ZZ^\ast,\ZZ} \\
+ \dual{D_z\II(t_{k},z_{k})-D_z\II(t_{k-1},z_{k})}{z_{k+1}-z_{k}}_{\ZZ^\ast,\ZZ}.
\end{multline*}
The $\kappa$-uniform convexity of $\II(t_k,\cdot)$ on $B_\ZZ(z(t_k),\Delta)$ thus gives 
$0 \geq \kappa \norm{z_{k+1}-z_{k}}_\ZZ^2 - |\ell|_{Lip} (t_{k}-t_{k-1}) \norm{z_{k+1}-z_{k}}_\VV$, which implies 
\[ \norm{z_{k+1}-z_k}_\ZZ \leq |\ell|_{Lip}/\kappa \, \tau \leq \frac{\kappa-\delta}{\kappa} \tau < \tau \] by the assumption on $|\ell|_{Lip}$. By the time-update \eqref{eq:tupdate}, we consequently have 
\begin{equation}\label{eq:auxA13}
t_k-t_{k-1} \geq \delta/\kappa \, \tau,
\end{equation}
which gives the well-posedness of our interpolant and the boundedness of its derivative in $\ZZ$ due to \cref{lem:bound_discrete-loc-Z}. From this Lemma and again the choice of $\tau$, we moreover conclude for $t \in [t_k,t_{k+1}]$
\begin{align*}
\norm{\tilde{z}_\tau(t)-z(t)}_\ZZ &\leq \norm{z_k-z(t_k)}_\ZZ + \norm{z(t_k)-z(t)}_\ZZ + \frac{t-t_k}{t_{k+1}-t_k}\norm{z_{k+1}-z_k}_\ZZ \\
&\leq K^\prime \, \tau + \norm{z}_{Lip} (t_{k+1}-t_k) + C_{loc} \tau 
\leq \Delta/3 + \Delta/3 + \Delta/3 = \Delta.
\end{align*}
Hence $\tilde{z}_\tau(t) \in B_\ZZ(z(t),\Delta)$ for all $t \in [t_0,t_{k+1}]$ so that the uniform convexity of $\II(t,\cdot)$ 
on $B_\ZZ(z(t), \Delta)$ implies that \eqref{eq:gammaest} holds on $[t_0,t_{k+1}]$.
Thus we can again argue as in the proof of \cref{thm:reverse_approx_small} (steps 2--6) to show 
\eqref{eq:eq.aux009} on the extended time interval $[t_0,t_{k+1}]$.
In summary, we therefore have shown that \eqref{eq:induction1}--\eqref{eq:eq.aux009} holds with $k+1$ instead of $k$, 
which completes the induction step. Hence, iterating this yields $\|\tilde z_\tau(t) - z(t)\|_\ZZ \leq K' \,\tau$ on the whole 
time interval $[0,T]$.\\[0.5ex]
\textsl{Step 3: Comparing Interpolants}\\
We again define the affine interpolant $\hat{t}_\tau$ as in \eqref{eq:affine-interpolants}. From \eqref{eq:auxA13}, it follows that $\hat{t}^\prime_\tau \geq \delta/\kappa$ for all $ s \in [0,S_\tau]$. Thus, there exists a unique inverse function $s_\tau : [0,T] \to [0,\hat{S}_\tau]$ with $ 1 \leq s_\tau^\prime(t) \leq \frac{1}{1- \frac{\kappa-\delta}{\kappa}}$ a.e. in $[0,T]$. 
In full analogy to the proof of \cref{thm:reverse_approx_small} (step 7), we obtain 
$\norm{z_\tau(t)-\tilde{z}_\tau(t)}_\ZZ \leq \tau$, 
where again $z_\tau$ is the retransformed affine interpolation, i.e. $z_\tau(t) := \hat{z}_\tau(s_\tau(t))$. Thus we finally get 
\[ \norm{z_\tau(t)-z(t)}_\ZZ \leq \norm{z_\tau(t)-\tilde{z}_\tau(t)}_\ZZ + \norm{\tilde{z}_\tau(t)-z(t)}_\ZZ \leq K_{loc} \, \tau , \]
which was claimed.
\end{proof}

\section{Numerical tests}\label{sec:tests}
In the next subsections, we provide two numerical examples in order to illustrate the theoretical findings of the previous section. 

\subsection{Globally uniformly convex energy}
We start with an infinite-dimensional example. For that, we let $\Omega = [0,1]^2$ and choose 
\[ \II(t,z) = \frac{1}{2} \dual{A z}{z}_{\ZZ^\ast,\ZZ} - \dual{\ell(t)}{z}_{\VV} \]
with $A = -\Delta : H^1_0(\Omega) \mapsto	H^{-1}(\Omega)$ and 
$\ell(t,x) = \mathds{1}_\Omega - \frac{1}{\pi} \cos(\pi \, t/2 ) f(x)$, 
wherein $f(x) = 2(x_1(1-x_1)+x_2(1-x_2))$. Moreover, the dissipation functional is given by the $L^1$-norm, i.e., $\RR(v) = \norm{v}_{L^1(\Omega)}$. Consequently, the underlying spaces are $\ZZ = H^1_0(\Omega)$, $\VV=L^2(\Omega)$, and $\XX = L^1(\Omega)$. In this setting, the unique (differential) solution to \eqref{eq:subDiffInc} reads
\begin{equation}\label{eq:eq.exactPDE_sol}
z(t,x) = \left\{
\begin{array}{rc}
0 & ,t  \in [0,1) \\
-\frac{1}{\pi} \cos(\frac{\pi}{2} t) \, v(x) & ,t \in [1,2) \\
-\frac{1}{\pi} \,  v(x) & ,t \in [2,3]
\end{array} \right.
\end{equation}
with $v(x) = x_1 x_2 (1-x_1)(1-x_2)$. For the spatial discretization of this system, we choose linear finite elements on a Friedrich-Keller triangulation 
with mesh size $h = \sqrt{2}/100$ and use a mass-lumping scheme for the discretization of $\RR$. 
The detailed implementation is described in \citep{fem_paramsol}. 
The resulting errors are shown in Figure~\ref{fig:pdeExample_rates}. 
It can be seen that the error decreases in a linear fashion (w.r.t. the time-parameter $\tau$) until the error of the spatial-discretization is dominating.
\begin{figure}[h!]
    \centering
      \includegraphics[scale=0.5]{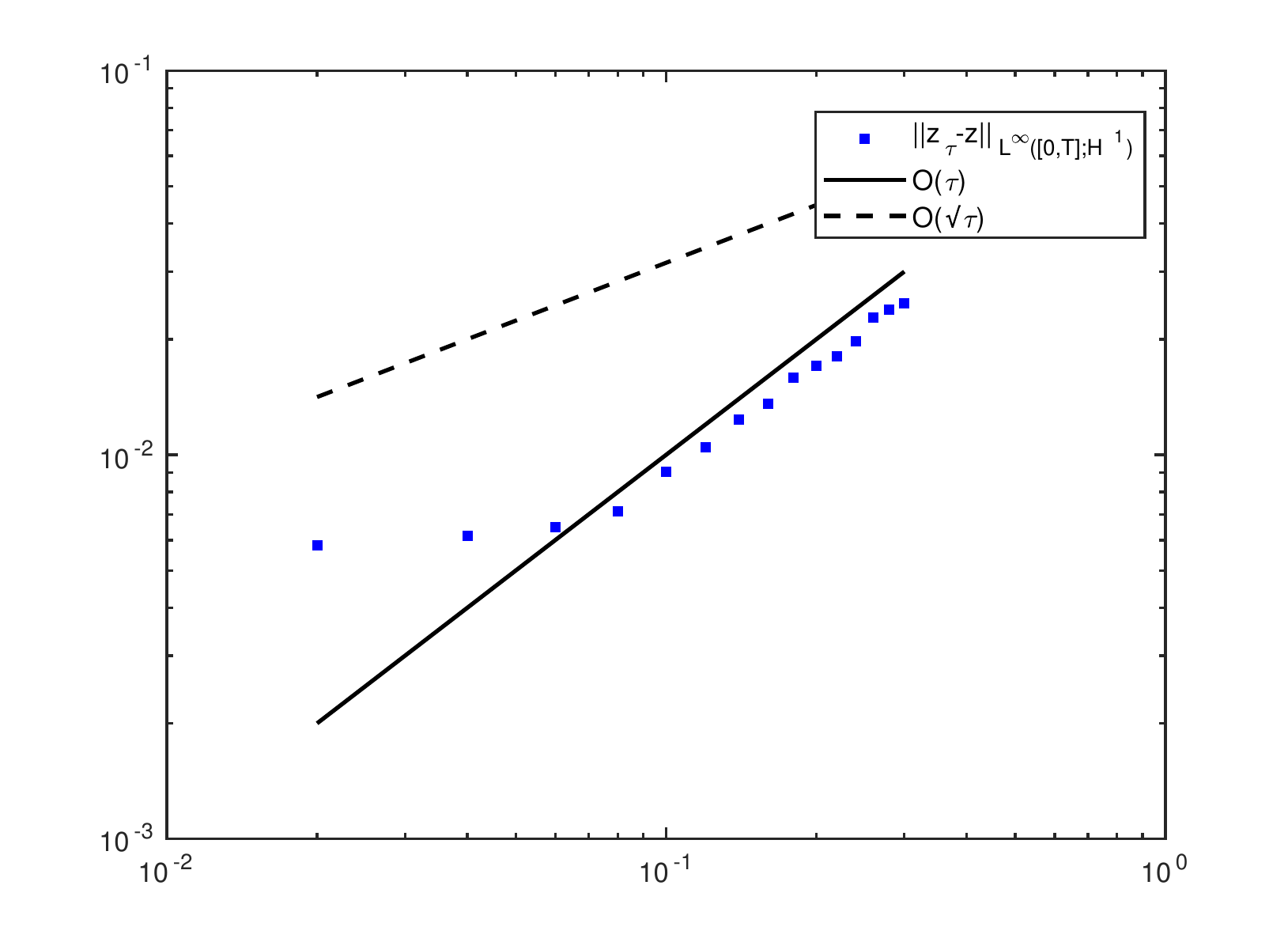} 
      \caption{Errors for the approximation of the parameterized solution \eqref{eq:eq.exactPDE_sol} using the local minimization scheme.} 
      \label{fig:pdeExample_rates} 
   \centering 
\end{figure}

\subsection{Locally uniformly convex energy}\label{sec:testloc}
We next give a one-dimensional example, in which the energy is not globally uniformly convex. 
In particular, the energetic solution will no longer be continuous in time, which is seen in Figure~\ref{fig:localConvex}. 
However, the parametrized solution is still Lipschitz-continuous and moreover remains in a region, where the energy is uniformly convex, 
see Figure~\ref{fig:localConvex}. For this example, we set $\ZZ=\VV=\XX=\R$ as well as:
\begin{align}
\RR(v) = \abs{v} \quad \text{and} \quad \II(t,z) = \frac{1}{2} z^2 + \FF(z) - \ell(t)z \label{eq:Energ_locConv_1D}
\end{align}
with
\begin{align*}
&\FF(z) = \left\{ \begin{array}{rc}
2z^3 - 5/2\, z^2 +1 & , z \geq 0 \\
-2z^3 - 5/2 \, z^2 +1 & , z<0
\end{array} \right. \quad \text{and} \quad  \ell(t) = -1/2(t-3/2)^2+3/2 .
\end{align*}
For $z_0 = -2/3$, a (differential) solution to \eqref{eq:subDiffInc} with \eqref{eq:Energ_locConv_1D} reads
\begin{equation} \label{eq:z_loc_analytic}
z(t) = \left\{
\begin{array}{rl}
-2/3 & ,t  \in [0,1/2) \\
-\frac{1}{3}(1+1/2 \sqrt{1+3(t-3/2)^2})  & ,t \in [1/2,2) \\
-1/2  & ,t \in [2,3]
\end{array} \right.
\end{equation}
By direct calculations, one verifies that $z$ indeed stays in a region, where $\II$ is uniformly convex. 
Thus, from the analysis in Section~\ref{sec:AprioriEstimates}, we expect the error in the approximation to be of order $\Landau{\tau}$, 
which can be nicely observed in the Figure~\ref{fig:localConvex} below. 
\begin{figure}[h!]
   \begin{subfigure}[t]{0.49\textwidth} 
      \includegraphics[width=\textwidth]{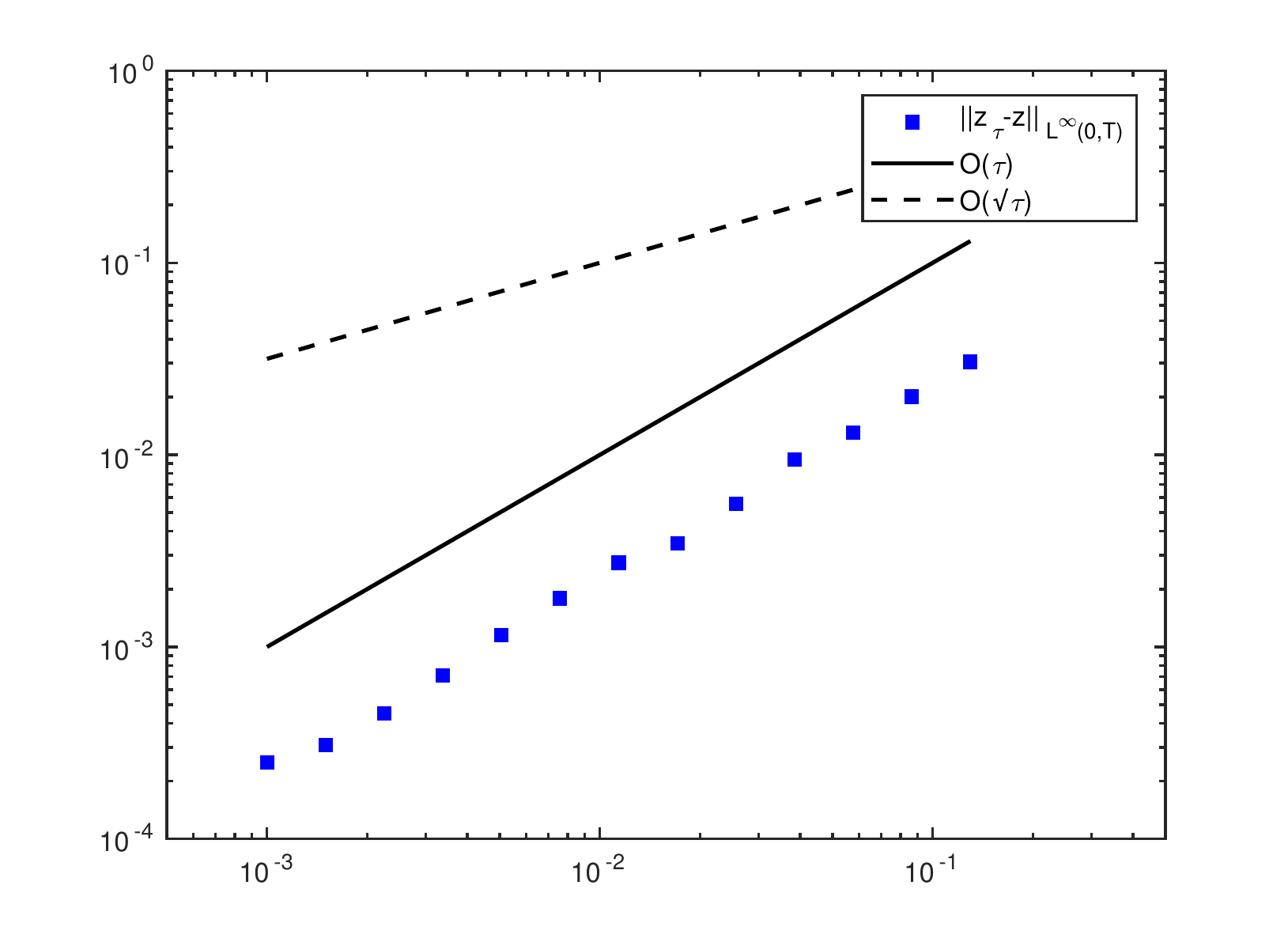} 
   \end{subfigure}\hfill%
   \begin{subfigure}[t]{0.49\textwidth} 
      \includegraphics[width=\textwidth]{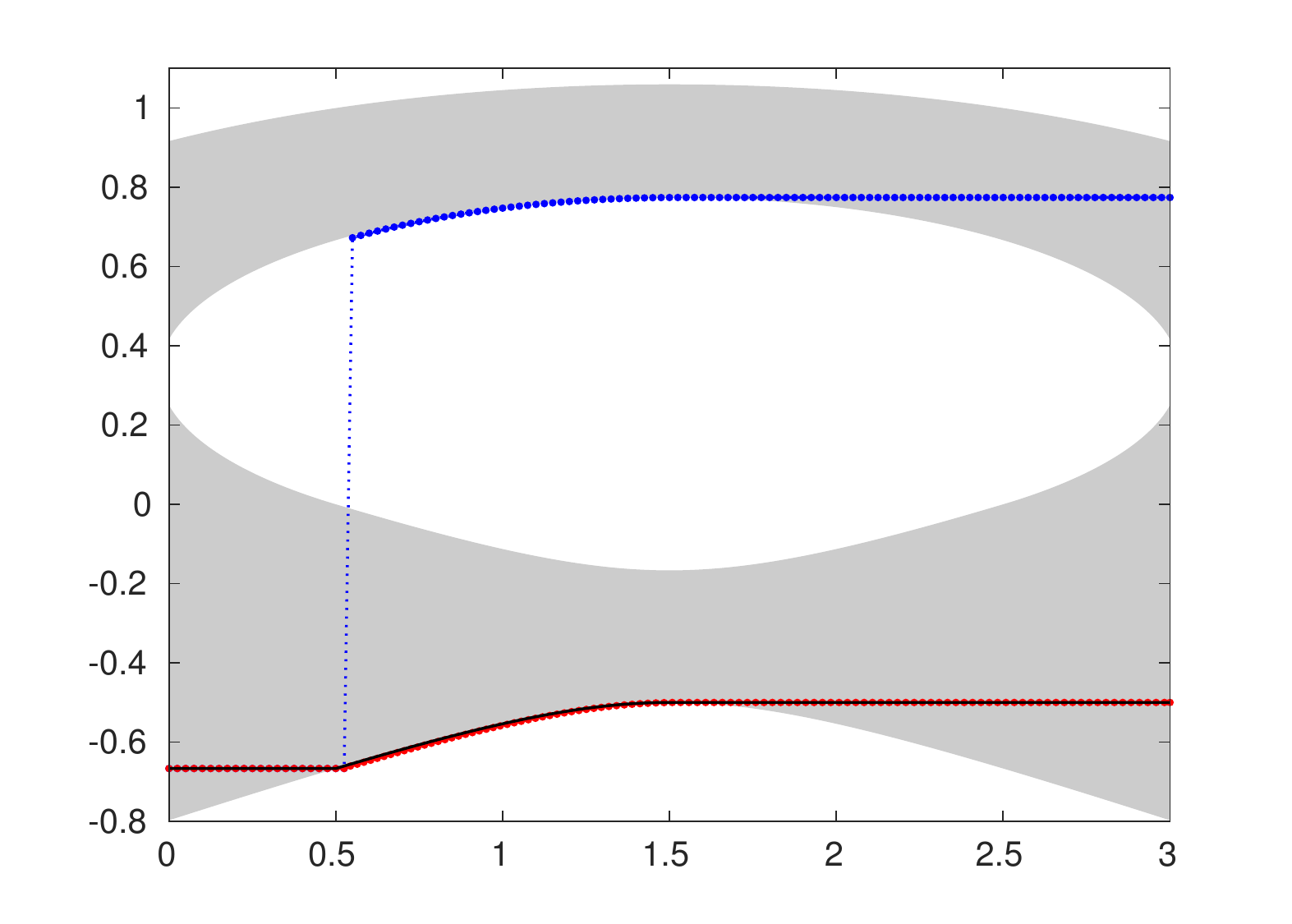}  
   \end{subfigure}\\[5pt]%
   \centering 
   \caption{Left: Errors for the approximation of a parametrized solution using the local minimization scheme depending on the stepsize $\tau$; Right: Corresponding differential solution (black) as well as the numerical approximations using the global (blue) and the local iterated minimization scheme (red) as functions of the time $t$.}
   \label{fig:localConvex}
\end{figure}

\begin{appendix}

\section{Estimation of the error measure $\gamma$}\label{sec:App.errorGamma}
In the proofs of \cref{thm:reverse_approx_small} and \cref{thm:reverse_approx_general}, we use an adapted version of an estimate that is part of the proof of uniqueness for solutions of RIS from \cite{mielketheil}. For convenience of the reader, we present this adapted version here. Therefor let $z_1,z_2 \in W^{1,\infty}([0,T];\ZZ)$ and again $\gamma(t) := \dual{D_z\II(t,z_1(t)) - D_z\II(t,z_2(t))}{z_1(t)-z_2(t)}_{\ZZ^\ast,\ZZ}$. First of all we calculate
\begin{align*}
\dot{\gamma}(t) 
&= \dual{D^2_z\II(t,z_1(t))[z_1(t)-z_2(t)]}{z_1^\prime(t)}_{\ZZ^\ast,\ZZ} 
- \dual{D^2_z\II(t,z_2(t))[z_1(t)-z_2(t)]}{z_2^\prime(t)}_{\ZZ^\ast,\ZZ} \\
&\quad + \dual{D_z\II(t,z_1(t)) - D_z\II(t,z_2(t))}{z_1^\prime(t)-z_2^\prime(t)}_{\ZZ^\ast,\ZZ} ,
\end{align*}
where we used the symmetry of $D_z^2\II$. Note that, due to the special structure of $\II$, the partial derivative w.r.t. $t$ is equal to zero. Rearranging terms, we arrive at
\begin{align*}
\dot{\gamma}(t) &= \dual{D^2_z\II(t,z_1(t))[z_1(t)-z_2(t)]+D_z\II(t,z_2(t))-D_z\II(t,z_1(t))}{z_1^\prime(t)}_{\ZZ^\ast,\ZZ} \\
&\quad - \dual{D^2_z\II(t,z_2(t))[z_1(t)-z_2(t)]+D_z\II(t,z_1(t))-D_z\II(t,z_2(t))}{z_2^\prime(t)}_{\ZZ^\ast,\ZZ} \\
&\quad + 2 \dual{D_z\II(t,z_1(t)) - D_z\II(t,z_2(t))}{z_1^\prime(t)-z_2^\prime(t)}_{\ZZ^\ast,\ZZ} \\
\end{align*} 
Now, due to $z_1,z_2 \in W^{1,\infty}([0,T];\ZZ)$ and the regularity on $\II(t,\cdot)$ (see \eqref{ass:regularity_I}), we find that 
\begin{align*}
\dot{\gamma}(t)
&\leq C \norm{z_1(t)-z_2(t)}_\ZZ^2 \norm{z_1^\prime(t)}_\ZZ 
+ C \norm{z_1(t)-z_2(t)}_\ZZ^2 \norm{z_2^\prime(t)}_\ZZ \\
&\quad + 2 \dual{D_z\II(t,z_1(t)) - D_z\II(t,z_2(t))}{z_1^\prime(t)-z_2^\prime(t)}_{\ZZ^\ast,\ZZ} \\
&\leq C \norm{z_1(t)-z_2(t)}_\ZZ^2 + 2 \dual{D_z\II(t,z_1(t)) - D_z\II(t,z_2(t))}{z_1^\prime(t)-z_2^\prime(t)}_{\ZZ^\ast,\ZZ}
\end{align*} 
which is the desired estimate.

\section{Existence and Uniqueness of differential solutions}\label{sec:diffexistence}
The statements of \cref{thm:reverse_approx_small} and \cref{thm:reverse_approx_general} each refer to the unique differential solution of \eqref{eq:subDiffInc}, which exists due to \citep[Thm. 7.4]{mielketheil}. However, in \citep{mielketheil}, the energy functional is assumed to be slightly more regular than as in \eqref{ass:regularity_I}. For completeness, we therefor bring together the necessary results from the literature to obtain the existence and uniqueness of differential solutions in our setting. 

\begin{theorem}\label{thm:app.existence_and_uniqueness}
Let $\II$ fulfill Assumption~\ref{ass:globconv}, i.e., it is $\kappa$-uniformly convex . 
Then there exists a unique differential solution $z \in W^{1,\infty}(0,T;\ZZ)$, i.e. it holds
\begin{equation}
0 \in \partial\RR(z^\prime(t)) + D_z\II(t,z(t)) \quad \faa t \in [0,T]. \label{eq:auxA11}
\end{equation}
\end{theorem}
\begin{proof}
First of all, the existence of a differential solution satisfying $ z \in W^{1,\infty}(0,T;\ZZ)$ follows from \citep[Cor. 3.4.6(i)]{mielkeroubi} combined with \citep[Cor. 3.1.2]{mielkeroubi}. Moreover, since $\II(t,\cdot)$ is uniform convex, every differential solution has to fulfill $z \in W^{1,\infty}(0,T;\ZZ)$ as a result of \citep[Thm. 3.4.4]{mielkeroubi}  (with $\alpha=2$, $\beta=1$) and \citep[Cor. 3.4.6(i)]{mielkeroubi}. 
Now, let $z_1,z_2 \in W^{1,\infty}(0,T;\ZZ)$ be two differential solutions. We again define $\gamma(t) := \dual{D_z\II(t,z_1(t)) - D_z\II(t,z_2(t))}{z_1(t)-z_2(t)}$.  Since $z^\prime \in L^1([0,T];\ZZ)$, \eqref{eq:auxA11} is equivalent to 
\begin{equation}\label{eq:app.uniqueDiffSol_reformulations}
    \RR(z^\prime(t)) \geq \RR(v) + \dual{-D_z\II(t,z(t))}{v-z^\prime(t)}_{\ZZ^\ast,\ZZ} \quad \forall v \in \ZZ . 
\end{equation}
Testing this variational inequality for $z_1$ with $z_2$ and vice versa and adding up the resulting inequalities, we obtain 
\[ 0 \geq \dual{D_z\II(t,z_1(t))-D_z\II(t,z_2(t))}{z_1^\prime(t)-z_2^\prime(t)}_{\ZZ^\ast,\ZZ} . \]
Exploiting the estimate from Section~\ref{sec:App.errorGamma}, we thus have
$\dot{\gamma}(t) 
\leq C \norm{z_1(t)-z_2(t)}_\ZZ^2$.
The $\kappa$-uniform convexity of $\II$ implies $\gamma(t) \geq \kappa \norm{z_1(t)-z_2(t)}_\ZZ^2$, so that $\dot{\gamma}(t) \leq C \gamma(t)$ and we obtain the uniqueness result by applying the Gronwall-Lemma. 
\end{proof}



\end{appendix}

\renewcommand{\refname}{References} 
\bibliography{database-references}

\begin{thebibliography}{10}

\bibitem{AC:NumAnElastoPlast}
{\sc J.~Alberty and C.~Carstensen}, {\em Numerical analysis of time-depending
  primal elastoplasticity with hardening}, SIAM Journal on Numerical Analysis,
  37 (2000), pp.~1271--1294.

\bibitem{acfs17}
{\sc M.~{Artina}, F.~{Cagnetti}, M.~{Fornasier}, and F.~{Solombrino}}, {\em
  {Linearly constrained evolutions of critical points and an application to
  cohesive fractures.}}, {Math. Models Methods Appl. Sci.}, 27 (2017),
  pp.~231--290, \url{https://doi.org/10.1142/S0218202517500014}.

\bibitem{bartels:errorEst}
{\sc S.~Bartels}, {\em Quasi-optimal error estimates for implicit
  discretizations of rate-independent evolutions}, SIAM Journal on Numerical
  Analysis, 52 (2014), pp.~708--716.

\bibitem{efenmielke06}
{\sc M.~A. Efendiev and A.~Mielke}, {\em On the rate-independent limit of
  systems with dry friction and small viscosity}, Journal of Convex Analysis,
  13 (2006), pp.~151--167.

\bibitem{hanreddy}
{\sc W.~Han and B.~Reddy}, {\em Plasticity}, Springer, New York, 1999.

\bibitem{knees17}
{\sc D.~Knees}, {\em Convergence analysis in time-discretization schemes for
  rate-independent systems}.
\newblock submitted to ESAIM:COCV, 2017,
  \url{https://arxiv.org/pdf/1712.06851.pdf}.

\bibitem{ks13}
{\sc D.~Knees and A.~Schr{\"o}der}, {\em Computational aspects of quasi-static
  crack propagation}, Discrete and Continuous Dynamical Systems. Series S, 6
  (2013), pp.~63--99, \url{https://doi.org/10.3934/dcdss.2013.6.63}.

\bibitem{fem_paramsol}
{\sc C.~Meyer and M.~Sievers}, {\em Finite element discretization of local
  minimization schemes for rate-independent evolutions}, Calcolo, 56 (2019),
  \url{https://doi.org/10.1007/s10092-018-0301-4}.

\bibitem{mielke:ERIS}
{\sc A.~Mielke}, {\em Chapter 6 evolution of rate-independent systems},
  Handbook of Differential Equations, Evolutionary Equations, 2 (2006).

\bibitem{mielke11differential}
{\sc A.~Mielke}, {\em Differential, energetic, and metric formulations for
  rate-independent processes}, Nonlinear PDE’s and Applications, 2028 (2011),
  pp.~87--167.

\bibitem{mpps:errorRIS}
{\sc A.~Mielke, L.~Paoli, A.~Petrov, and U.~Stefanelli}, {\em Error estimates
  for space-time discretizations of a rate-independent variational inequality},
  SIAM Journal on Numerical Analysis, 48 (2010), pp.~1625--1646.

\bibitem{mrs12}
{\sc A.~Mielke, R.~Rossi, and G.~Savar{\'e}}, {\em B{V} solutions and viscosity
  approximations of rate-independent systems}, ESAIM. Control, Optimisation and
  Calculus of Variations, 18 (2012), pp.~36--80,
  \url{https://doi.org/10.1051/cocv/2010054}.

\bibitem{mrs16}
{\sc A.~{Mielke}, R.~{Rossi}, and G.~{Savar\'e}}, {\em {Balanced viscosity (BV)
  solutions to infinite-dimensional rate-independent systems.}}, {J. Eur. Math.
  Soc. (JEMS)}, 18 (2016), pp.~2107--2165,
  \url{https://doi.org/10.4171/JEMS/639}.

\bibitem{mielkeroubi}
{\sc A.~Mielke and T.~Roub\'{i}\u{c}ek}, {\em Rate-Independent Systems: Theory
  and Application}, Springer-Verlag, New York, 2015.

\bibitem{mielketheil}
{\sc A.~Mielke and F.~Theil}, {\em On rate-independent hysteresis models},
  NoDEA: Nonlinear Differential Equations and Applications, 11 (2004),
  pp.~151--189.

\bibitem{Neg14}
{\sc M.~Negri}, {\em Quasi-static rate-independent evolutions:
  characterization, existence, approximation and application to fracture
  mechanics}, ESAIM Control Optim. Calc. Var., 20 (2014), pp.~983--1008,
  \url{https://doi.org/10.1051/cocv/2014004}.

\bibitem{RinSchwarzSueli17}
{\sc F.~Rindler, S.~Schwarzacher, and E.~Süli}, {\em Regularity and
  approximation of strong solutions to rate-independent systems}, Mathematical
  Models and Methods in Applied Sciences, 27 (2017), pp.~2511--2556,
  \url{https://doi.org/10.1142/S0218202517500518}.

\end{thebibliography}

\end{document}